\documentclass{amsart}
\usepackage[capitalize,nameinlink,noabbrev,nosort]{cleveref}

\usepackage{amssymb, amsmath, amsfonts,mathrsfs,graphicx ,mathdots, amscd}
\usepackage{color}
\usepackage{mathtools}
\usepackage{lineno}
\usepackage{tikz}
\usepackage{tikz-cd}
\usepackage{quiver}
\usepackage{lipsum}
\usepackage{adjustbox}
\usepackage{dynkin-diagrams}
\def\row#1/#2!{#1_{\IfStrEq{#2}{}{n}{#2}} & \dynkin{#1}{#2}\\}

 \usepackage{youngtab}
\usepackage[boxsize=.3 em]{ytableau}

\usepackage{verbatim}

\usepackage[all,cmtip]{xy}
\usepackage{stmaryrd}

\numberwithin{equation}{section}

\theoremstyle{plain}
\newtheorem{theorem}{Theorem}[section]

\newtheorem{prop}[theorem]{Proposition}
\newtheorem{cor}[theorem]{Corollary}
\newtheorem{lemma}[theorem]{Lemma}
\newtheorem{conj}[theorem]{Conjecture}
\newtheorem{mainprop}{Proposition}

\newtheorem{mainprob}{Problem}

\newtheorem{mainthm}{Theorem}

\newtheorem{conjecture}[theorem]{Conjecture}

\theoremstyle{definition}
\newtheorem{defn}[theorem]{Definition}
\newtheorem{example}[theorem]{Example}

\theoremstyle{remark}
\newtheorem{remark}[theorem]{Remark}

\newcommand{\circled}[1]{{\raisebox{.5pt}{\textcircled{\raisebox{-.9pt} {#1}}}}}

\newcommand{\CC}{{\mathbb C}}

\newcommand{\ZZ}{{\mathbb Z}}
\newcommand{\QQ}{{\mathbb Q}}
\newcommand{\XX}{{\mathbb X}}

\newcommand{\checkXX}{{\check{\mathbb X}}}
\newcommand{\Pcal}{{\mathcal P}}
\newcommand{\Qcal}{{\mathcal Q}}
\newcommand{\height}{\operatorname{ht}}

\newcommand{\RpwP}{\check R^{+}_{w,P}}

\newcommand{\RpwB}{\check R^{+}_{w,B}}

\newcommand{\Pic}{\operatorname{Pic}}

\newcommand{\checkalpha}{{\check \alpha}}
\newcommand{\checketa}{{\check \eta}}
\newcommand{\checkmu}{{\check \mu}}

\renewcommand{\div}{\operatorname{div}}

\newcommand{\inv}{^{-1}}

\begin{document}{\allowdisplaybreaks[4]}

\title[An anticanonical perspective on $G/P$ Schubert varieties]{An anticanonical perspective on $G/P$  Schubert varieties}


\author{Changzheng Li}
 \address{School of Mathematics, Sun Yat-sen University, Guangzhou 510275, P.R. China}
\email{lichangzh@mail.sysu.edu.cn}

\author{Konstanze Rietsch}
\address{ Department of Mathematics, King's College London, Strand, London WC2R 2LS, UK}
\email{konstanze.rietsch@kcl.ac.uk}

\author{Mingzhi Yang}
 \address{School of Mathematics, Sun Yat-sen University, Guangzhou 510275, P.R. China}
\email{yangmzh8@mail2.sysu.edu.cn}

\thanks{}

\date{
      }




\begin{abstract}
We describe a natural basis of the Cartier class group of an arbitrary Schubert variety $\XX_{w,P}$ in a flag variety $G/P$ of general Lie type. We  then characterise  when the Schubert variety is factorial/Fano, along with an explicit formula for the anticanonical line bundle in these cases.  We also  prove that, for Schubert varieties in simply-laced types (only), being  factorial is equivalent to being $\QQ$-factorial, and is equivalent to the equality of the Betti numbers $b_2(\XX_{w,P})=b_{2\ell(w)-2}(\XX_{w,P})$. Finally, we give a convenient characterisation of when a simply-laced Schubert variety is Gorenstein and when it is Gorenstein Fano.
\end{abstract}
\maketitle
\setcounter{tocdepth}{1} 
\tableofcontents
\section{Introduction}
\vskip .2cm

Let $G$ be a simply-connected simple complex algebraic group together with a choice of upper-triangular Borel subsgroup $B$ and parabolic $P $ with $B\subseteq P$. Our main object of study is the Schubert variety
\[
\XX_{w,P}:=\overline{BwP/P}
\]
associated to a Weyl group element $w$, where we choose $w$ to be a minimal coset representative,  $w\in W^P$, for the associated parabolic subgroup of the Weyl group $W$. The variety $\XX_{w,P}$ is an irreducible $\ell(w)$-dimensional subvariety of $G/P$. Moreover, it is well-known that $\XX_{w,P}$ is normal, Cohen-Macaulay and has rational singularities \cite{Ram,  Andersen, LaSe}.
More recently, Woo and Yong  characterised when type $A$ Schubert varieties $\XX_{w,B}$ satisfy local properties  that measure the singularities in \cite{WooYong08}. They gave a combinatorial characterisation of when a type $A$ Schubert variety is Gorenstein or Fano in \cite{WooYong}, and posed a problem there and again  in \cite[Problem 8.24]{WooYong23}. Namely
\begin{mainprob}\label{problem}
Determine which Schubert varieties $\XX_{w,P}$ in $G/P$ are ($\mathbb{Q}$-)Gorenstein/($\mathbb{Q}$-)factorial.
\end{mainprob}
\noindent The Schubert variety $\XX_{w,P}$ comes with a natural anticanonical divisor. In this paper we describe when this divisor is Cartier, and in that case express its associated line bundle in the Picard group in a natural way. Along the way we provide an answer to Problem \ref{problem}. The answer is particularly nice in the simply-laced case, where we show that  $\QQ$-Gorenstein is equivalent to Gorenstein and $\QQ$-factorial to factorial. As a result we prove a  conjecture of Enomoto  \cite[Conjecture A.8]{Emoto} in simply-laced types, but we disprove the conjecture in general.

\subsection{ } We have two natural sets of divisors in $\XX_{w,P}$. On the one hand, we have the Schubert divisors $\XX_{x,P}\subset \XX_{w,P}$, which are given by elements $x\prec w$, namely $x\in W^P$ with $x< w$ and $\ell(x)=\ell(w)-1$. We recall that the Schubert divisors determine a basis of $H_{2\ell(w)-2}(\XX_{w,P},\ZZ)$. Or equivalently, they form a basis of the (Weil) divisor class group $A_{\ell(w)-1}(\XX_{w,P})$. The Chow groups of a Schubert variety agree with the integral homology groups \cite{FMSS}.

On the other hand we have the projected Richardson varieties. Namely, in $G/B$ we have the `Richardson varieties' of $\XX_{w,B}$ associated to elements $v\le w$ of $W$ by
\[
\XX^{v}_{w,B}=\overline{B_-v B\cap BwB/B}.
\]
These are reduced  irreducible subvarieties of dimension $\ell(w)-\ell(v)$ by \cite{KaLu}. If $v\not< w$ the intersection above is empty. Amongst these we have the codimension one Richardson varieties corresponding to simple reflections $s_i$ appearing in $w$. The $G/P$ analogue is the `projected Richardson variety', and we obtain divisors in $\XX_{w,P}$ of the form
\[
\XX^{s_i}_{w,P}:=\pi_P(\XX^{s_i}_{w,B}),
\]
where $\pi_P:G/B\to G/P$ is the natural projection. Let us write $I^B_w:=\{i\in I\mid s_i<w\}$ for the set indexing the Richardson divisors in $\XX_{w,B}$. This set simultaneously indexes also the projected Richardson varieties in $\XX_{w,P}$. We have a natural subset $I^P_w:=\{k\in I^B_w\mid s_k\in W^P\}$.
The role of the divisors $\XX^{s_k}_{w,P}$ is elucidated by the following result, which forms the starting point for our paper.
\begin{mainprop}\label{mt:PicardGroup}
The projected Richardson divisor $\XX^{s_k}_{w,P}$ in $\XX_{w,P}$ associated to $k\in I^P_w$ is Cartier, and the Picard group $\operatorname{Pic}(\XX_{w,P})$  is freely generated by the line bundles $\mathcal O(\XX^{s_k}_{w,P})$ with $k\in I^P_w$.
\end{mainprop}

By \cite{Mat88} the line bundles on $\XX_{w,P}$ all arise as restrictions of line bundles on $G/P$. From this perspective the line bundle $\mathcal O(\XX^{s_k}_{w,P})$ from the proposition is just isomorphic to the restriction of the line bundle $\mathcal L_{\omega_k}$ associated to the fundamental weight $\omega_k$ via the Borel-Weil construction. See also \cite{WooYong,RW:Schub} for such statements in the $G/B$ and the Grassmannian Schubert variety case, respectively.

Our perspective is now that the first set of divisors, the set of Schubert divisors $\XX_{x,P}$, gives a basis of the Weil divisor class group, while the second set of divisors, $\XX^{s_k}_{w,P}$, gives a basis of the Cartier divisor class group. Moreover, the former group is naturally isomorphic to  $H_{2\ell(w)-2}(\XX_{w,P},\ZZ)$ and the latter group can be identified with $H^2(\XX_{w,P},\ZZ)$.
Our approach will be to  use the relationship between these two sets of divisors to analyse properties of $\XX_{w,P}$ such as the property of factoriality or of being Gorenstein.

\subsection{}
Our first application involving \cref{mt:PicardGroup}  is about characterising factorial  Schubert varieties. Let $b_i$ denote the $i^{\rm th}$ Betti number, so  $b_i(\XX_{w,P})=\dim H_{i}(\XX_{w,P},\ZZ)$.
We recall the work of Carrell, Kuttler and Peterson \cite{CP,CK} which proves that in simply-laced type $\XX_{w,B}$ is smooth if and only if its Poincar\'e polynomial is palindromic, namely $b_i(\XX_{w,P})=b_{2\ell(w)-i}(\XX_{w,P})$ for all $i$. Below we provide a similar characterization for simply-laced Schubert varieties with smoothness replaced by factoriality.
\begin{mainthm}\label{mt:simplylaced-factorial}
In  simply-laced types, a Schubert variety $\XX_{w,P}$ is factorial if and only if $b_2(\XX_{w,P})=b_{2\ell(w)-2}(\XX_{w,P})$. Moreover, any $\QQ$-Cartier divisor in $\XX_{w,P}$ is automatically Cartier.
\end{mainthm}
  \begin{remark}\label{r:Qfactorial}
For arbitrary Schubert varieties $\XX_{w,P}$, the condition $b_2(\XX_{w,P})=b_{2\ell(w)-2}(\XX_{w,P})$ is equivalent to being $\QQ$-factorial. This follows from the fact that Schubert varieties are normal varieties with rational singularities
using a general result of Park and Popa~\cite[Theorem A]{ParkPopa}. From this perspective, our result says that $\QQ$-factoriality is equivalent to factoriality for simply-laced Schubert varieties. The second statement of \cref{mt:simplylaced-factorial} also implies that being $\QQ$-Gorenstein is equivalent to being Gorenstein for simply-laced Schubert varieties.
\end{remark}
\begin{remark}
Our results have further parallels to the results of Carrell, Kuttler and Peterson. Namely, in \cite{CP} it is proved that $\XX_{w,B}$ having a palindromic Poincar\'e polynomial is equivalent to it being rationally smooth, which can also be interpreted in terms of the Kazhdan-Lusztig polynomial \cite{KaLu} as saying $P_{e,w}=1$. Then  \cite{CK} shows that in simply-laced type being rationally smooth is equivalent to being smooth. Here in this work we consider $\QQ$-factorialty of $\XX_{w,P}$, which is equivalent $b_2=b_{2\ell(w)-2}$ by \cite{ParkPopa}, which can further be restated as vanishing of the degree $1$ term of $P_{e,w}$ by \cite[Theorem D]{BE09}. And we prove that for $\XX_{w,P}$ of simply-laced type, being $\QQ$-factorial is equivalent to being factorial.
\end{remark}

\begin{remark}\label{r:historyFactorial} For the complete flag variety $G/B$ of type $A$, the characterisation of factoriality given in \cref{mt:simplylaced-factorial} also follows from an earlier combinatorial characterisation for  Schubert varieties to be factorial that is found in work of \cite{BMB07,Emoto}.
Thus in type $A$ it was already known that factoriality is equivalent to $b_2(\XX_{w,P})=b_{2\ell(w)-2}(\XX_{w,P})$. Our \cref{mt:simplylaced-factorial} extends this result to all simply-laced types with a new and uniform proof. This verifies for simply-laced Schubert varieties a conjecture made by Enomoto in  \cite[Conjecture~A.8]{Emoto}. However, we  give counterexamples to the conjecture of Enomoto in the non-simply-laced case.
\end{remark}

Underlying \cref{mt:simplylaced-factorial} is another result that characterises factoriality for general Schubert varieties. In \cref{d:R+wP} we associate a set of coroots $\RpwP$ to $\XX_{w,P}$, consisting of positive coroots   $\checketa$ such that $ws_{\checketa}\in W^P$ has length $\ell(w)-1$.  We characterise this set $\RpwP$ entirely in terms of root system combinatorics. For example, $\RpwB$ is the set of positive coroots sent to negative that are \textit{indecomposable}, see  \cref{d:indecomposableV2} and  \cref{p:indecomposableV2}. For the additional condition to lie in the parabolic subset $\RpwP$, see \cref{p:RwP}.  The general criterion for factoriality in these terms is as follows.

\begin{mainthm} \label{mt:factorialgen}
A Schubert variety $\XX_{w,P}$ of arbitrary Lie type is factorial if and only if the integer matrix $M=\big(\langle\omega_k, \checketa\rangle\big)_{\checketa\in\RpwP,\,k\in I_w^P}$ is square and has $\det(M)=\pm 1$.
\end{mainthm}

\noindent Note that the row indexing set $\RpwP$ of the matrix $M$ from the theorem is in bijection with the set of Schubert divisors in $\XX_{w,P}$, and the column indexing set with the Cartier divisors $\XX^{s_k}_{w,P}$, so that the matrix $M$ has dimensions $b_{2\ell(w)-2}\times b_2$. Therefore $\QQ$-factoriality is equivalent to the matrix being square, see \cref{r:Qfactorial}.
From this perspective, \cref{mt:simplylaced-factorial} says that in simply-laced types, as soon as the matrix is square it will automatically be invertible over $\ZZ$.

One ingredient in the proof of these theorems is showing that the matrix $M$ from \cref{mt:factorialgen} always has maximal rank, for any $P$ and $w\in W^P$. This is another application of \cref{mt:PicardGroup}, and we note that as a purely  combinatorial statement about root systems it is rather non-trivial. We also note that using \cref{mt:factorialgen}, it becomes  straightforward to determine whether a given Schubert variety is factorial via root system combinatorics.

\subsection{}
The next main object of interest is the anticanonical class, and determining whether individual Schubert varieties are Fano. For this we consider the total sum of \textit{all} the projected Richardson divisors $\XX^{s_i}_{w,P}$ and all the Schubert divisors $\XX_{x,P}$ in a given Schubert variety. This sum is an anticanonical divisor in $\XX_{w,P}$, which was implicit in  \cite{ramanathan, KLS} and we prove in \cref{p:BasicKX} following     \cite[Proposition 2.2 and Lemma 3.8]{LSZ}.  We write
\begin{eqnarray*}
-K_{\XX_{w,P}}&=&\sum_{i\in I^B_w} \XX^{s_i}_{w,P}\ +\ \sum_{x\prec w}  \XX_{ x,P}.
\end{eqnarray*}

Let us now consider the setting of \textit{factorial} Schubert varieties $\XX_{w,P}$ and characterise which of these have ample anticanonical divisor. Associate to $\XX_{w,P}$ the (in this case, square) matrix $M_{w,P}=(\langle\omega_k,\checketa\rangle)_{\checketa,k}$  Let $b=b_2(\XX_{w,P})=|I^P_w|$. By ordering rows and columns appropriately we may assume that $M_{w,P}\in SL_b(\ZZ)$, see \cref{mt:factorialgen}.  We have the following result.

\begin{mainthm}\label{mt:factorialFano} Suppose $\XX_{w,P}$ is a factorial Schubert variety of arbitrary Lie type. Let $N_{w,P}$ be the inverse of the  matrix $M_{w,P}$ associated to $\XX_{w,P}$ above and let $\mathbf 1$ denote the constant vector $(1)_{\checketa\in \RpwP}$ and  $\mathbf h=(h_\checketa)_{\checketa\in \RpwP}$ the vector of heights $h_\checketa=\height(\checketa)$. Let $N_{w,P}(\mathbf h+\mathbf 1)=(\hat n_k)_{k\in I^P_w}$.  Then the anticanonical class of $\XX_{w,P}$ is given by
\[
[-K_{\XX_{w,P}}]=\sum_{k\in I^P_w} \hat n_k [\XX^{s_k}_{w,P}],
\]
in terms of the projected Richardson divisor classes from \cref{mt:PicardGroup}. Moreover, the factorial Schubert variety $\XX_{w,P}$ is Fano if and only if the vector $N_{w,P}(\mathbf h+\mathbf 1)$ is positive, that is, all $\hat n_k>0$.
\end{mainthm}
In terms of the Borel-Weil description of line bundles,   we have that the anticanonical bundle of $\XX_{w,P}$ is the restriction of the line bundle $\mathcal L_{\lambda}$ on $G/P$ where $\lambda=\sum_{k\in I^P_w} \hat n_k\omega_k$.

\subsection{} Let us now assume $\XX_{w,P}$ is a Schubert variety of simply-laced type (with no factoriality assumption). Recall that elements of $H_2(\XX_{w,P},\ZZ)$ can be represented by coroots. If $\cap[\XX_{w,P}]:H^2(\XX_{w,P},\ZZ)\to H_{2\ell(w)-2}(\XX_{w,P},\ZZ)$ is the cap-product map, which here can also be interpreted as the map sending the Chern class of a line bundle to its divisor class, we consider the dual map
\[
(\cap[\XX_{w,P}])^*:H^{2\ell(w)-2}(\XX_{w,P},\ZZ)\longrightarrow H_2(\XX_{w,P},\ZZ).
\]
In simply-laced types we prove that one can choose a  subset of Schubert cohomology classes whose image under $(\cap[\XX_{w,P}])^*$ forms a (nonstandard) basis of $H_2(\XX_{w,P},\ZZ)$.  This basis is encoded in a subset $\mathcal B_{w,P}$ of coroots inside $\RpwP$ indexed by $I^P_w$ which we construct non-canonically.

In the case of a $G/B$ Schubert variety (still simply-laced), the element $\checkmu_w(i)\in \mathcal B_{w,B}$ associated to $i\in I^B_w$ is chosen from amongst positive coroots $\checkmu$ satisfying $w(\checkmu)<0$ and $\langle\omega_i,\checkmu\rangle> 0$, as indecomposable summand of the \textit{minimal} element in terms of a suitable `reflection ordering', see \cref{s:BwB}. The basis $\mathcal B_{w,P}$ for a general parabolic is constructed starting from the subset of $\mathcal B_{w,B}$ indexed by $I^P_w$ via a sequence of lemmas that culminate in \cref{d:BwP}. We do not repeat the construction in the introduction. However, the set $\mathcal B_{w,P}$ has the remarkable property of agreeing with the subset $\{\checkmu_w(k)\mid k\in I^P_w\}$ of $\mathcal B_{w,B}$ modulo $\langle\checkalpha_j\mid j\in I_P\rangle_\ZZ$. This property, in particular, is highly reliant on the simply-laced condition.

Let now $\XX_{w,P}$ be a Schubert variety of simply-laced type, and $M_{w,P}$ the $I^P_w\times I^P_w$-matrix
\[
M_{w,P}:=(\langle \omega_i,\checkmu_w(k)\rangle)_{k,i\in I^P_w},
\]
where $\mathcal B_{w, P}=\{\checkmu_w(k)\mid k\in I^P_w\}$ is the subset of $\RpwP$ described above.  As part of \cref{mt:Gorenstein}  below, $M_{w, P}$ constructed in this way will be invertible over $\ZZ$. We let $N_{w,P}$ denote the inverse matrix to $M_{w,P}$, and consider the vector $N_{w,P}(\mathbf h+\mathbf 1)=(\hat n_k)_{k\in I^P_w}$ in $\ZZ^{I^P_w}$, where $\mathbf 1=(1,\dotsc, 1)$, and $\mathbf h=(h_k)_{k\in I^P_w}$ is the vector of heights of the $\checkmu_w(k)$.
\begin{mainthm} \label{mt:Gorenstein}
Let $\XX_{w,P}$ be a simply-laced Schubert variety. Then the following statements hold.
\begin{enumerate}
\item The matrix $M_{w,P}$ is invertible over~$\ZZ$.
\end{enumerate}
\begin{enumerate}
\item[(2)] The Schubert variety $\XX_{w,P}$ is Gorenstein if and only if for every $\checketa\in\RpwP\setminus\mathcal B_{w,P}$ we have $\left\langle\sum_{k\in I^P_w}\hat n_k\omega_k\, ,\,\checketa\right\rangle  -\height(\checketa)=1$.
\item[(3)] Assuming $\XX_{w,P}$ is Gorenstein, we have
\[[-K_{\XX_{w,P}}]=\sum_{k\in I^P_w}\hat n_k[\XX^{s_k}_{w, P}],
\]
in terms of the basis of Cartier divisor classes from \cref{mt:PicardGroup}. Moreover, $\XX_{w,P}$ is Gorenstein Fano if and only if  $\hat n_k> 0$   for all $k\in I^P_w$.
\end{enumerate}

\end{mainthm}
For the first part of this theorem, that $M_{w,P}$ is invertible over $\ZZ$, we in fact show that this matrix is a unipotent triangular matrix for some ordering of $I^P_w$.
This result is also a key ingredient in the proof of \cref{mt:simplylaced-factorial}.  It is a special feature of the simply-laced case that cannot be extended to the non-simply-laced setting.

Note that the columns of the matrix $N_{w,P}$ may be naturally interpreted as elements of a dual basis to $\mathcal B_{w,P}$. Namely, $\checkmu_w^*(k)=\sum_{k'\in I^P_w} n_{k'k}\omega_{k'}$ is the dual basis element to $\checkmu_w(k)$. Furthermore, $\mathcal B^*_{w,P}=\{\checkmu_w^*(k)\mid k\in I^P_w\}$ can be interpreted as a basis of $H^2(\XX_{w,P},\ZZ)$. In these terms we have whenever $\XX_{w,P}$ is Gorenstein that
\begin{equation}\label{e:c1B}
c_1(\XX_{w,P})= \sum_{k\in I^P_w}\left(\height(\checkmu_w(k))+1\right) \checkmu^*_w(k),
\end{equation}
and if additionally $P=B$, then this simplifies to
\begin{equation}\label{e:c1P}
c_1(\XX_{w,B})= \sum_{i\in I^B_w}\omega_i + \sum_{i\in I^B_w} \checkmu^*_w(i),
\end{equation}
which,  if $I^B_w=I$,  is just the $\rho$-shift of the sum of the $\checkmu^*_w(i)$. Moreover, in the simply-laced $G/B$ Schubert variety case, the condition to be Gorenstein from \cref{mt:Gorenstein} can be reformulated as
\begin{equation}\label{e:IntroBGorenstein}
\text{$\XX_{w,B}$ is Gorenstein} \quad \iff\quad  \langle\sum_{i\in I^B_w} \checkmu^*_w(i),\checketa\rangle=1 \text{ for all $\checketa\in \RpwB\setminus\mathcal B_{w,B}$.}
\end{equation}

We finally remark that it was shown in \cite{KWY} that local properties of Richardson varieties $\XX_{w, P}^{u, P}$ can be reduced to that of
the Schubert varieties $\XX_{w, P}$ and $\XX^{u, P}$. Our theorems above may have further applications in the study of the singularities of Richardson varieties.
\subsection{} The paper is organised as follows. After introducing notations we define the important set $\RpwP$ of positive coroots associated to a Schubert variety $\XX_{w,P}$.  In \cref{s:RpwP} we first characterise the set $\RpwB$ in terms of root system combinatorics as set of `indecomposable' elements in $\check R^+(w)$, with special consideration of the simply-laced case. See \cref{p:indecomposableV2} and \cref{c:RwBsl}. These results are extended  in \cref{p:RwP} to a characterisation of $\RpwP$ for general $P$. Next, in \cref{s:Picard}, we describe the Picard group $\Pic(\XX_{w,P})$, identifying it with a subgroup of the Weil divisor class group. This is where \cref{mt:PicardGroup} is proved. We obtain a characterisation of factorial and $\QQ$-factorial Schubert varieties in terms of $\RpwP$ in \cref{s:factorial}. The main results in this section are \cref{p:Qcal} and \cref{t:Factorialsl}, which imply both \cref{mt:simplylaced-factorial} and \cref{mt:factorialgen}. The proof of \cref{t:Factorialsl} involves the  construction of a subset $\mathcal B_{w,P}$ in $\RpwP$ for any simply-laced Schubert variety $\XX_{w,P}$. The construction of this subset is carried out in \cref{s:BwB} (Borel case) and \cref{s:BwP} (extension to arbitrary parabolic $P$). In \cref{s:Fano} we make use of this set $\mathcal B_{w,P}$ to determine when a simply-laced Schubert variety is Gorenstein  and  in that case compute its first Chern class. In that section we also compute the first Chern class for any factorial Schubert variety in  general type, see \cref{p:anticanon} and \cref{p:anticanon2}. We then  characterise which of these Schubert varieties are Fano in \cref{c:Fano}. These propositions prove \cref{mt:Gorenstein} and \cref{mt:factorialFano}. We provide a wealth of examples in \cref{s:Examples}. We also formulate three conjectures related to Weyl group combinatorics that are stated at the end of \cref{s:RpwP} and \cref{s:BwB}.

\subsection*{Acknowledgements}
The authors  thank Hongsheng Hu, Haidong Liu, Nick Shepherd-Barron, Lauren Williams, David Sheard, Rui Xiong and Alexander Yong for helpful  conversations. C. Li  is  supported   in part by the National Key Research
and Development Program of China No. 2023YFA1009801 and NSFC Grant   12271529.

 \section{Notation}\label{s:notation}
In this section,  we review some relevant background from Lie theory and algebraic geometry, and refer to   \cite{BL:Book,Bour, Springer:Book, BjornerBrentiBook, Brion, Kollar, Kumar} for more details.
\subsection{Notations in Lie theory}
Let $G$ be a simply-connected simple complex algebraic group, $B$ be a Borel subgroup of $G$, and $P\supseteq B$ be a parabolic subgroup of $G$. Let $B_-$ denote the Borel subgroup opposite to $B$, and then $T:=B\cap B_-$ is a maximal torus of $G$ with Lie algebra $\mathfrak{t}=\mbox{Lie}(T)$.
Let $\Pi=\{\alpha_i\mid i\in I\}\subset \mathfrak{t}^*$ (resp.  $\check \Pi=\{\checkalpha_i\mid i\in I\}\subset \mathfrak{t})$
  denote the set of simple roots (resp. coroots) inside the positive roots  $R^+$
  (resp. coroots  $\check R^+$). Let $\{\omega_i\mid i\in I\}\subset \mathfrak{t}^*$ denote the  fundamental weights. Denote  $\rho:=\sum_{i\in I}\omega_i$.

  We express elements  of the Weyl group $W$ in terms of the simple reflection generators $W=\langle s_i\mid i\in I\rangle$. We also take a lifting $\dot s_i$ in the normalizer $N_G(T)$ of $T$, so that $s_i\mapsto \dot s_i T$ defines a (canonical) isomorphism $W\cong N_G(T)/T$. Associate to $P$, we have a subset $I_P:=\{i\in I\mid \dot s_i T\subset P\}$, and denote $I^P:=I\setminus I_P$.
  Let  $W_P=\langle s_i\mid i\in I_P\rangle$ denote the Weyl subgroup of $W$ associated to $P$, and $\ell: W\to \mathbb{Z}_{\geq 0}$ denote the standard length function. Then we have the subset $W^P\subset W$, consisting  of minimal-length coset representatives for $W/W_P$.  We note that $w\in W^P$ if and only if $w(\alpha_j)>0$ for all $j\in I_P$. Let $(W, <)$ denote the Bruhat order on $W$. Namely for $u, w\in W$, we have $u<w$ if some substring of a reduced expression for $w$ gives a reduced expression for $u$.  We further write
    $u\prec w$ if $w$ covers $u$ in the Bruhat order, that is, if $u<w$ and $\ell(u)=\ell(w)-1$.

 For $w\in W$, we denote by  $\check R^+(w)$  the inversion set of coroots, and write out its elements using a reduced expression $w=s_{i_1}\dotsc s_{i_r}$,
   \[
\check R^+(w)=\{\checkalpha\in \check R^+\mid w(\checkalpha)<0\}= \{\checkalpha_{i_r}, s_{i_r}(\checkalpha_{i_{r-1}}),\dotsc, s_{i_r}s_{i_{r-1}}\dotsc s_{i_2}(\checkalpha_1)\}.
\]
There is a total ordering $<_{\mathbf i}$ of $\check R^+(w)$ determined by the chosen reduced expression of $w$, according to which the elements in $\check R^+(w)$ as listed above are in order. See \cite{Papi} for a full characterisation of these orderings. We only note that       the ordering $<_{\mathbf i}$ has the property that if $\checketa'<_{\mathbf i}\checketa''\in \check R^+(w)$ and $r',r''\in \mathbb R_{>0}$ are such that $\checketa=r'\checketa'+r''\checketa''$ lies in $\check R^+$, then $\checketa\in \check R^+(w)$ and $\checketa'<_{\mathbf i} \checketa<_{\mathbf i} \checketa''$.  An ordering with this property is called a \textit{ reflection ordering}, see~\cite{Dyer93}.
We may also write $\checketa^{\mathbf i}_{(k)}$ for the $k$-th element of this list, with $\mathbf i=(i_1,\dotsc,i_r)$ encoding the reduced expression.

\subsection{Geometry of flag varieties}  The flag variety $G/P$ is a smooth projective variety. Abuse the notation for $w$ with its image $\dot w T$ in $N_G(T)/T$. Inside $G/P$, there are Schubert varieties
   $$\XX_{w, P}:=\overline{BwP/P},\qquad  \XX^{v, P}:=\overline{B_-vP/P} $$
of dimension $\ell(w)$ and codimension $\ell(v)$ respectively, where $w, v\in W^P$. Their intersections $\XX^{v, P}_{w, P}=\XX_{w, P}\cap \XX^{v, P}$  are called Richardson varieties in $G/P$. Note that  $\XX^{v, P}_{w, P}$ is a reduced irreducible subvariety of dimension $\ell(w)-\ell(v)$ if $v\leq w$, or an empty set otherwise.
Let $\pi_P: G/B\to G/P$ denote the natural projection. We simply denote
$$\XX^{v}_{w, B}:=\XX^{v, B}_{w, B}, \qquad \XX_{w, P}^{v}:=\pi_P(\XX_{w, B}^{v})$$
the latter of which is called a projected Richardson variety in $G/P$, distinct from the Richardson variety $\XX^{v, P}_{w, P}$  in $G/P$ in general.
\begin{remark}\label{r:Richardson}
In $SL_3/P=\mathbb P^2$ where $I^P=\{1\}$ the projected Richardson variety $\XX^{s_2}_{s_2s_1, P}$ is already  not a Richardson variety in $\mathbb P^2$, see \cite[Example 1.23]{Speyer:Richardson}. Though it is still a divisor. However, when $k\in I_w^P$, we have
$$\XX^{s_k}_{w,P}=\pi(\XX^{s_i, B} \cap  \XX_{w, B})\subset \pi(\XX^{s_k, B})\cap \pi(\XX_{w, B})= \XX^{s_k, P}\cap \XX_{w, P}.$$
Both are irreducible and of the same dimension $\ell(w)-1$. Hence, they coincide with each other.
\end{remark}

 The Chow ring $A^*(G/P)$ is isomorphic to the integral cohomology $H^*(G/P, \mathbb{Z})$, which has a canonical  basis of  Schubert classes  $P.D.[\XX^{v, P}]\in H^{2\ell(v)}(G/P, \mathbb{Z})$. In particular, the Picard group $\Pic(G/P)=A^1(G/P)$ is given by $H^2(G/P, \mathbb{Z})$.

\subsection{Local properties}
In birational geometry, it is fundamental to study the local properties of a complex variety, such as whether it is factorial, Gorenstein, or Cohen-Macaulay among other properties. Schubert varieties are known to be normal and Cohen-Macaulay with rational singularities in complete generality. But they may or may not be Gorenstein or factorial.
\begin{defn}
A normal variety $X$ is called Gorenstein (resp. $\mathbb{Q}$-Gorenstein) if
its canonical divisor $K_X$ is Cartier (resp. $\QQ$-Cartier).
\end{defn}

\begin{defn}
A  variety $X$ is called factorial (resp. $\mathbb{Q}$-factorial) if
every Weil (resp. $\QQ$-Weil) divisor of $X$ is   Cartier (resp. $\QQ$-Cartier).
\end{defn}
We note that there are notions of Gorenstein/factorial defined by properties on the local ring $\mathcal{O}_{X, x}$   (see e.g. \cite{BrHe}). To define Gorenstein, $X$ is assumed to be Cohen-Macaulay instead of normal.  The two kinds of definitions  are equivalent when $X$ is normal and Cohen-Macaulay, in particular when $X=\XX_{w, P}$ is a Schubert variety.
These properties measure the singularities of $X$ with  the following relationships.
$$ \mbox{smooth}\ \ \subset\ \  \mbox{factorial}\ \ \subset \ \ \mbox{Gorenstein}\ \  \subset \ \ \mbox{normal } \& \mbox{ Cohen-Macaulay}\ \  \subset\ \  \mbox{arbitrary singularities}$$

\section{The set of coroots $\RpwP$}\label{s:RpwP}
In this section, we  define a special set of coroots $\RpwP$   associated to $\XX_{w,P}$. The heart of this section is Proposition \ref{p:indecomposableV2}, which  provides a characterisation of the special case $\RpwB$, and \cref{p:RwP} which adapts the characterisation of $\RpwB$ to general $\RpwP$. We also prove a key lemma, \cref{l:charWPcond}, which will play an important role again in \cref{s:BwP}.
\begin{defn}\label{d:R+wP}  For $w\in W^P$, we   define    a subset of $\check R^+(w)$ by
\begin{equation}\label{e:RpwP}
\begin{array}{ccl}
\RpwP &:=&\{\checketa\in \check R^+(w)\mid \ell(ws_\checketa)=\ell(w)-1, w s_\checketa\in W^P\}.
\end{array}
\end{equation}
\end{defn}

Note $W^B=W$. The special case,
\begin{equation}
\begin{array}{ccl}\label{e:RpwB}
\RpwB&=&\{\checketa\in \check R^+(w)\mid \ell(ws_\checketa)=\ell(w)-1\},
\end{array}
\end{equation}
is known as the  \emph{cover inversion set} of coroots for $w\in W$.

\begin{defn}\label{d:indecomposableV2}
We call an element $\checketa\in \check R^+(w)$
\begin{itemize}
\item \textit{simply decomposable} in $\check R^+(w)$, if $\checketa=\checkmu+\checkmu'$ for some $\checkmu,\checkmu'\in\check R^+(w)$,
\item \textit{decomposable} in
$\check R^+(w)$, if $\checketa=\frac{1}c(\checkmu+\checkmu')$ for two distinct $\checkmu,\checkmu'\in\check R^+(w)$ and $c\in\ZZ_{>0}$,
\item \textit{weakly decomposable} in $\check R^+(w)$, if $\checketa=r'\checkmu+r''\checkmu'$ for two distinct $\checkmu,\checkmu'\in\check R^+(w)$ and $r,r'\in \mathbb R_{>0}$.
\end{itemize}
We   say that $\checketa\in \check R^+(w)$ is \textit{indecomposable}, if it is not decomposable  in $\check R^+(w)$.
\end{defn}

\begin{lemma}\label{l:decomposablesl}
Suppose $G$ is simply-laced and let $w\in W$ and $\checketa\in \check R^+(w)$. If $\checketa=r\checkmu+r'\checkmu'$ for some $r,r'\in\mathbb R_{>0}$ and distinct $\checkmu,\checkmu'\in\check R^+(w)$, then we have the following.
\begin{enumerate}
\item $r=r'=1$, that is, $\checketa=\checkmu+\checkmu'$.
\item $\langle\eta,\checkmu\rangle=\langle\eta,\checkmu'\rangle=1$ and $\langle\mu,\checkmu'\rangle=-1$.
\end{enumerate}
 In particular, 
 the three notions of `decomposable' from \cref{d:indecomposableV2} are equivalent in simply-laced types.
\end{lemma}
\begin{proof}
We may assume $\checkmu'<_{\mathbf i} \checketa<_{\mathbf i} \checkmu$ for the reflection ordering associated to some reduced expression $\mathbf i$ of $w$. We may moreover assume that $\checkmu'=\checkalpha_i$, a simple coroot. Namely, if not then write $\checkmu'=s_{i_r}\dotsc s_{i_{\ell+1}}(\checkalpha_{i_\ell})$, factorise $w$ as
$w=w_1w_2=(s_{i_1}\dotsc s_{i_{\ell}})(s_{i_{\ell+1}}\dotsc s_{i_r})$, and apply $s_{i_{\ell+1}}\dotsc s_{i_r}$ to $\checketa=r\checkmu+r'\checkmu'$. We then obtain an identity of the same type, but now in $\check R^+(w_1)$ and with $\checkmu'$ replaced by $\checkalpha_{i_\ell}$.

We now assume $\checkmu'=\checkalpha_i$ and $\checketa=r\checkmu + r'\checkalpha_i$, with $r,r'\in\mathbb R_{>0}$.
For any positive coroot expanded in terms of the simple coroots $\checkalpha_j$ we make the following straightforward observation that holds in the simply-laced case.
\begin{center}
$(\star)\quad$ If a positive coroot is not simple, then at least \textit{two}
 of its coefficients will be equal to $1$.
\end{center}
We make use of this fact while comparing the two sides of the equation,
\begin{equation}\label{e:checketadecomp0}
\checketa-r'\checkalpha_i=r\checkmu.
\end{equation}
The left-hand side of \eqref{e:checketadecomp0} has the following properties.
\begin{enumerate}
\item[(i)] All coefficients are integral except for at most one: the coefficient of $\checkalpha_i$.
\item[(ii)] At least one coefficient is equal to $1$, by $(\star)$.
\end{enumerate}
For the right-hand side, observe that if $r$ is not an integer, then at least \textit{two} coefficients of $r\checkmu$ are not integers by $(\star)$, contradicting (i). Therefore $r$ must be an integer. Now if $r\ne 1$, then since $r$ is positive and an integer we have $r\ge 2$. This implies that every nonzero coefficient of $r\checkmu$ is $\ge 2$, in contradiction with (ii). Therefore it follows that $r=1$.

Next consider $r'\in\mathbb R_{>0}$. Since $\checketa=\checkmu+r'\checkalpha_i$ we have that $r'$ is a positive integer and
\[
\langle \alpha_i,\checketa\rangle=\langle\alpha_i,\checkmu\rangle + 2r'.
\]
Since in the simply-laced type the pairings take values in $\{\pm 1,0,2\}$,  the only solution is $\langle \alpha_i,\checketa\rangle=1, \langle \alpha_i,\checkmu\rangle=-1$ and $r'=1$. Both statements of the lemma  now follow.
\end{proof}
\begin{prop}\label{p:indecomposableV2}
Fix $w\in W$ and suppose $\checketa\in\check R^+(w)$. The following statements are equivalent.
\begin{enumerate}
\item $\ell(ws_\checketa)=\ell(w)-1$.
\item $\checketa$ is not decomposable in $\check R^+(w)$.
\end{enumerate}
The set $\RpwB$  is the subset of indecomposable elements in $\check R^+(w)$.
\end{prop}

The following corollary is simply the combination of \cref{p:indecomposableV2} and \cref{l:decomposablesl}.
\begin{cor}\label{c:RwBsl}
Suppose $G$ is simply-laced. Then $\checketa\in R^+(w)$ lies in the subset $R^+_{w,B}$ if and only if there is no decomposition $\checketa=\checkmu+\checkmu'$ with $\checkmu,\checkmu'\in R^+(w)$. \qed
\end{cor}

\begin{proof}[{Proof of \cref{p:indecomposableV2}}]
Recall that $\checketa\in\check R^+(w)$ implies $\ell(ws_\checketa)<\ell(w)$. We first prove the following statement.
\begin{center}
\noindent $(\star\star)$\quad If
$\ell(ws_\checketa)<\ell(w)-1$,
then $\checketa$ must be decomposable in $\check R^+(w)$.
\end{center}
\noindent This will prove that (2) $\implies$ (1).
\vskip .2cm

We use induction on $\ell(w)$ to prove $(\star\star)$. If $\ell(w)=1$ then the statement is empty. Now we may assume that the statement is proved whenever $w$ has length $\le r-1$. Let us suppose $\ell(w)=r$. We consider an element $\checketa\in \check R^+(w)$ for which $\ell(ws_\checketa)\le \ell(w)-2$.
Write $\checketa=s_{i_r}\dots s_{i_{\ell+1}}(\checkalpha_{i_\ell})$ using a reduced expression $\mathbf i=(i_1,\dotsc, i_r)$ of $w$. Then $ws_\checketa$ is obtained by removing the $s_{i_\ell}$ factor from $w$ where $1< \ell<r$, and we may write
\begin{equation}\label{e:what}
\hat w:=ws_\checketa=(s_{i_1}\dotsc s_{i_{\ell-1}})(s_{i_{\ell+1}}\dotsc s_{i_r}),
\end{equation}
keeping in mind, though, that the above is no longer a reduced expression. Let us also consider the element $v:=s_{i_1}w$. The element $v$ has the following immediate properties,
\begin{equation}\label{e:v}
v=s_{i_2}s_{i_3}\dotsc s_{i_r}, \qquad \ell(v)=r-1, \qquad \checketa\in \check R^+(v)=\check R^+(w)\setminus\{s_{i_r}\dotsc s_{i_2}(\alpha_{i_1})\}.
\end{equation}
By the induction hypothesis, if $\ell(vs_\checketa)\le \ell(v)-2$ then we must have that $\checketa$ is decomposable in $\check R^+(v)$. Therefore by \eqref{e:v} it is also  decomposable in $\check R^+(w)$ and we are done. So we can now assume that $\ell(vs_\checketa)=\ell(v)-1$. Then in summary we have
\begin{equation}\label{e:vproperties}
\ell(vs_\checketa)=\ell(v)-1=r-2, \qquad
\ell(ws_{\checketa})=\ell(s_{i_1}vs_{\checketa})\overset{(\dagger)}=\ell(vs_{\checketa})-1=r-3.
\end{equation}
Note that $\ell(s_{i_1}vs_{\checketa})=\ell(vs_{\checketa})\pm 1$, and the equality $(\dagger)$ above follows from our assumption that $\ell(ws_\checketa)\le \ell(w)-2$. We now have as reduced expression for $vs_\checketa$ the following,
\[
\hat v:=vs_\checketa\overset{\text{red}}{=\joinrel=}s_{i_2}s_{i_3}\dotsc s_{i_{\ell-1}} s_{i_{\ell+1}}\dotsc s_{i_r}.
\]
Additionally, the equality $\ell(s_{i_1}\hat v)=\ell(\hat v)-1$  implies that $\hat v\inv(\checkalpha_{i_1})<0$. We now consider the following two positive coroots,
\begin{eqnarray}\label{e:mumu'}
\check\mu&:=&\hat w\inv(\checkalpha_{i_1})=\hat v\inv s_{i_1}(\alpha_{i_1})>0,\\
\check\mu'&:=&w\inv s_{i_1}(\checkalpha_{i_1})=s_{i_r}\dotsc s_{i_2}(\alpha_{i_1})>0.\notag
\end{eqnarray}
Using $\hat w\inv=s_{\checketa} w\inv $, we compute their sum as follows,
\begin{eqnarray}\label{e:mu+mu'}
\check\mu+\check\mu' &= & s_\checketa w\inv (\checkalpha_{i_1}) - w\inv (\checkalpha_{i_1})\\ &=& \left(w\inv (\checkalpha_{i_1}) - \langle{\eta, w\inv (\checkalpha_{i_1})}\rangle\,  \checketa\right) - w\inv (\checkalpha_{i_1})
=  - \langle{\eta, w\inv (\checkalpha_{i_1})}\rangle \, \checketa.\notag
\end{eqnarray}
Since $\check\mu,\check\mu'$ and $\checketa$ are all positive coroots, we must have that $- \langle{\eta, w\inv (\checkalpha_{i_1})}\rangle =c>0$ (with $c=1$ in the simply-laced case). We now rewrite \eqref{e:mu+mu'} as
\begin{equation}\label{e:checketadecomp}
\checketa=\frac{1}{c}(\check\mu+\check\mu').
\end{equation}
Note that $\check\mu'$ lies in $\check R^+(w)$ by the second expression for it given in \eqref{e:mumu'}. Let us check that also $\check\mu\in \check R^+(w)$. Using again that $\hat w\inv=s_{\checketa} w\inv $, we have
\begin{eqnarray}\label{e:wmu}
w(\check\mu)&=& w s_{\checketa}w\inv (\checkalpha_{i_1})=s_{w(\checketa)}(\checkalpha_{i_1})=
\checkalpha_{i_1}-\langle{w(\eta), \checkalpha_{i_1}}\rangle w(\checketa)
\\
&=&\checkalpha_{i_1}-\langle{\eta, w\inv( \checkalpha_{i_1})}\rangle w(\checketa)=\checkalpha_{i_1}+ c \, w(\checketa).\notag
\end{eqnarray}
Since $\checketa\in\check R^+(w)$ we have $w(\checketa)<0$, and since $c>0$ it follows that $\checkalpha_{i_1}+ c \, w(\checketa)$ cannot be a positive coroot. Thus we deduce from \eqref{e:wmu} that $w(\check\mu)$ is a negative coroot and therefore $\check\mu\in \check R^+(w)$. Now \eqref{e:checketadecomp} demonstrates that $\checketa$ is decomposable in $\check R^+(w)$. This completes the proof the statement $(\star\star)$.

We have now proved (2) $\implies$ (1).
It remains to show (1) $\implies$ (2).
Suppose $\ell(ws_\checketa)=\ell(w)-1$. Thus now the expression in \eqref{e:what} is a reduced expression for $ws_\checketa$.
Let us assume indirectly that $\checketa=\frac{1}c(\checkmu+\checkmu')$
 for  some $c\in \mathbb R_{>0}$ and distinct elements  $\checkmu,\checkmu'\in \check R^+(w)$.   Then in any reflection ordering $<_\mathbf i$ of $\check R^+(w)$ we have that $\checketa$ appears in the middle between the $\checkmu,\checkmu'$. We may assume that
\begin{equation}\label{e:mulessmu'}
\checkmu'<_{\mathbf i}\checketa<_{\mathbf i}\checkmu.
\end{equation}
By truncating $w$ at either end we may additionally assume $\checkmu=s_{i_r}\dotsc s_{i_2}(\checkalpha_{i_1})$ and $\checkmu'=\checkalpha_{i_r}$, compare the proof of \cref{l:decomposablesl}. Let us write $i={i_r}$ for simplicity. We have $\checketa=\frac 1 c(\checkalpha_i+\checkmu)$. Applying $s_\eta$ to both sides gives
\[
-\checketa=\frac 1c(s_{\checketa}(\checkalpha_i)+ s_\checketa s_{i_r}\dotsc s_{i_2}(\checkalpha_{i_1}))=
\frac 1c\left (s_{\checketa}(\checkalpha_i)+ s_{i_r}\dotsc s_{i_{\ell+1}} s_{i_{\ell-1}}\dotsc s_{i_2}(\checkalpha_{i_1})\right).
\]
By our assumption that $\ell(ws_\checketa)=\ell(w)-1$ we have that the right-hand summand in the brackets above is an element of $R^+(w s_\checketa)$. In particular it is positive. Rearranging, we have
\[
  c\,\checketa +  s_{i_r}\dotsc s_{i_{\ell+1}} s_{i_{\ell-1}}\dotsc s_{i_2}(\checkalpha_{i_1})=-s_\checketa(\checkalpha_i)=\langle \eta,\checkalpha_i\rangle\checketa - \checkalpha_i,
\]
and therefore
\[
\checkalpha_i+  s_{i_r}\dotsc s_{i_{\ell+1}} s_{i_{\ell-1}}\dotsc s_{i_2}(\checkalpha_{i_1})= \left(\langle\eta,\checkalpha_i\rangle -c\right)\checketa.
\]
Thus, since the left-hand side is a sum of positive roots, this would imply that
\begin{equation}\label{e:cinequality}
\langle\eta,\checkalpha_i\rangle >c.
\end{equation}
On the other hand, it follows from  $\ell(ws_\checketa)=\ell(w)-1$   that $\checkalpha_i\in \check R^+(ws_\checketa)$. Therefore
\begin{eqnarray}\label{e:wschecketa}w s_\checketa(\checkalpha_i)&=&w\left(\checkalpha_i-\frac{1}c\langle \eta, \checkalpha_i\rangle(\checkmu+\checkalpha_i)\right)
= \left(1-\frac 1c\langle \eta,\checkalpha_i\rangle \right ) w(\checkalpha_i)-\frac 1c\langle \eta,\checkalpha_i\rangle w(\checkmu).\end{eqnarray}
must be negative. But by \eqref{e:cinequality} and since $\checkalpha_i,\checkmu\in \check R^+(w)$, we have that both summands on the right-hand side are positive. Thus we have obtained a contradiction.
\end{proof}

\begin{lemma}\label{l:charWPcond} Let $w\in W^P$ and suppose that $\checketa\in \check R^+(w)$ and $j\in I_P$.
Then $ws_\checketa(\checkalpha_j)<0$ if and only if $s_\checketa(\checkalpha_j)\in  \check R^+(w)$. Moreover
\[
ws_\checketa\in W^P \iff s_\checketa(\checkalpha_j)\in \check R\setminus \check R^+(w)\quad \text{for all $j\in I_P$.}
\]
\end{lemma}
\begin{proof}
If    $s_\checketa(\checkalpha_j)\in \check R^+(w)$, then clearly $ws_\checketa(\checkalpha_j)<0$, by the definition of $\check R^+(w)$.

 Now we assume $ws_\checketa(\checkalpha_j)<0$. If $s_\checketa(\checkalpha_j)<0$, then $\langle \eta, \checkalpha_j \rangle>0$,   because
\begin{equation*}\label{e:setaalapha}
s_\checketa(\checkalpha_j)=\checkalpha_j-\langle \eta, \checkalpha_j\rangle \checketa
\end{equation*}
would otherwise describe $s_\checketa(\checkalpha_j)$ as sum of two positive roots. Since $w\in W^P$,   $w(\checkalpha_j)>0$. Since $\checketa\in R^+(w)$,   $w(\checketa)<0$. Hence,
$ws_{\checketa}(\checkalpha_j)=w(\checkalpha_j)+(-\langle \eta, \checkalpha_j \rangle  w(\checketa))>0$, resulting in a contradiction. Hence,  $s_\checketa(\checkalpha_j)>0$, so that
 it is in  $R^+(w)$. Hence, the first statement follows.

 Note that $ws_\checketa\in W^P$ if and only if $ws_\checketa(\checkalpha_j)>0$ for all $j\in I_P$.  The second statement is therefore a direct consequence of the first statement.
\end{proof}

 In   \cref{p:indecomposableV2}, we have described the set $\RpwB$ directly as natural subset of the set $\check R^+(w)$, without any further reference to $w$ or its reduced expressions. Using  \cref{l:charWPcond}, we can immediately extend this description to describe the subset $\RpwP$ for general $P$ as follows.
\begin{prop}\label{p:RwP} For any parabolic $P$ and $w\in W^P$ we have
\[
\RpwP=\{\checketa\in\check R^+(w)\mid \checketa \text{ is indecomposable in $\check R^+(w)$ and } s_\checketa(\checkalpha_j)\notin \check R^+(w) \text{ whenever $j\in I_P$}\}. \qed
\]
\end{prop}

\begin{example}\label{ex:G2} Consider $G$ of type $G_2$, with long root $\alpha_1$ and short root $\alpha_2$.
\begin{center}\dynkin[edge length=.8cm,
root radius=.09cm,label]{G}{2}$\qquad \langle\alpha_1,\checkalpha_2\rangle=-3,\quad \langle\alpha_2,\checkalpha_1\rangle=-1$.
\end{center}
There are two reflection orderings on $R^+=R^+(w_0)$ corresponding to $w_0=s_1 s_2 s_1s_2s_1s_2=s_2 s_1s_2s_1s_2 s_1$, or $\mathbf i=(1,2,1,2,1,2)$ and $\mathbf i'=(2,1,2,1,2,1)$. Namely, these are
\[
\checkalpha_{2}\ <_{\mathbf i}\  \checkalpha_1+\checkalpha_2\ <_{\mathbf i} \ 3\checkalpha_1+ 2\checkalpha_2\  <_{\mathbf i}\  2\checkalpha_1+\checkalpha_2  \ <_{\mathbf i}\  3\checkalpha_1+\checkalpha_2\ <_{\mathbf i}\ \checkalpha_1
\]
and its reverse,
\[
\checkalpha_{1}\ <_{\mathbf i'}\  3\checkalpha_1+\checkalpha_2\ <_{\mathbf i'} \ 2\checkalpha_1+ \checkalpha_2\  <_{\mathbf i'}\  3\checkalpha_1+2\checkalpha_2  \ <_{\mathbf i'}\  \checkalpha_1+\checkalpha_2\ <_{\mathbf i'}\ \checkalpha_2.
\]
If $w$ is the length $k$ element of $W$ that ends $s_2$ (resp. $s_1$), then we can read off $\check R^+(w)$ as the first $k$ elements of $\check R^+(w)$ according to $\mathbf i$ ( resp. $\mathbf i'$).   Moreover, the indecomposable elements are precisely the first and the last in the ordered set $\check R^+(w)$ (as can be verified by inspection). These indeed coincide with the elements of the subset $\RpwB$.

For example, if $w_1=s_2 s_1 s_2 s_1$, then we have
\[
\check R^+_{w_1,B}=\{\checkalpha_{1}, \ 3\checkalpha_1+2\checkalpha_2\} \subset \check R^+(w_1)=\{\checkalpha_{1},\ \checkalpha_1+\checkalpha_2,\ 2\checkalpha_1+ \checkalpha_2,\    3\checkalpha_1+2\checkalpha_2\},
\]
while for $w_2=s_2 s_1 s_2 s_1 s_2$ we have
\[
\check R^+_{w_2,B}=\{\checkalpha_{2}, \ 3\checkalpha_1+\checkalpha_2\} \subset R^+(w_2)=\{\checkalpha_{2},\ \checkalpha_1+\checkalpha_2,\ 3\checkalpha_1+ 2\checkalpha_2,\    2\checkalpha_1+\checkalpha_2, 3\checkalpha_1+\checkalpha_2\},
\]
These are the indecomposable elements of $R^+(w_1)$ and $R^+(w_2)$, respectively.

Note that we also have $w_1\in W^{P_1}$ and $w_2\in W^{P_2}$ where $P_i$ denotes the standard parabolic subgroup with $I^{P_i}=\{i\}$.
Applying the additional condition $w_is_\checketa\in W^{P_i}$ to $\checketa \in \check R^+_{w_i,B}$ gives
\[
\check R^+_{w_1,P_1}=\{3\checkalpha_1 + 2\checkalpha_2\},\qquad \check R^+_{w_2,P_2}=\{3\checkalpha_1 + \checkalpha_2\}.
\]
\end{example}
\begin{remark}
 By \cref{ex:G2}, in type $G_2$,  the integer $c$ arising in   \cref{d:indecomposableV2} may be $1,2$ or $3$. In types $B_n, C_n$ and $F_4$, the integer $c$ may be $1$ or $2$.
\end{remark}

We end this section with the following two conjectures motivated by \cref{p:indecomposableV2}.
\begin{conjecture}\label{c:RootConjecture}
Suppose $W$ is simply-laced, and $w\in W$ with $\check R^+(w)$ its associated inversion set. If $\checkalpha\in\check R^+(w)$ is decomposable, then there \textit{exists} a decomposition $\checkalpha=\checkmu+\checkmu'$ such that for two different reduced expressions $\mathbf i,\mathbf i'$ of $w$ we have $\checkmu<_{\mathbf i}\checkmu'$, but $\checkmu'<_{\mathbf i'}\checkmu$.
\end{conjecture}
We note that the statement does not necessarily hold for every decomposition of $\checkalpha$, but the conjecture says it holds for at least one, see \cref{ex:Conj1} below.
Using \cref{p:indecomposableV2} the above conjecture can be reduced to the following conjecture that is stated directly in terms of the Coxeter group structure of $W$.
\begin{conjecture}\label{c:CoxeterConj}
Let $w\in W$ and suppose $t\in W$ is a reflection such that $\ell(wt)<\ell(w)-1$. Then there exists a reduced expression $w=s_{i_1}\dotsc s_{i_{\ell-1}}s_{i_\ell}s_{i_{\ell+1}}\dotsc s_{i_r}$ with $i_{\ell-1}=i_{\ell+1}=i$ such that
\[
wt=s_{i_1}\dotsc s_{i}\hat {s}_{i_\ell}s_{i}\dotsc s_{i_r},
\]
where the factor $s_{i_\ell}$ has been removed.
\end{conjecture}
Note that this implies the first conjecture as follows. In the situation of \cref{c:RootConjecture}, if $\checkalpha\in \check R^+(w)$ is decomposable then $t=s_{\checkalpha}$ has the property $\ell(wt)<\ell(w)-1$ by \cref{p:indecomposableV2}. Thus \cref{c:CoxeterConj} implies the existence of a reduced expression $\mathbf i=(i_1,\dotsc,i,i_\ell,i\dotsc i_r)$ where $\checkalpha=s_{i_r}\dotsc s_{i_{\ell+2}}s_{i}(\checkalpha_{j})$, writing $j=i_\ell$. This gives a substring in the associated reflection ordering that is of the form $\checkmu,\checkalpha,\checkmu'$ with $\checkalpha=\checkmu+\checkmu'$, which is reversed in the reflection ordering of $\mathbf i'=(i_1,\dotsc j,i,j,\dotsc i_r)$, obtained from $\mathbf i$ by applying a single braid relation.

\begin{example}\label{ex:Conj1}.
 Consider the minimal coset representative $w^{P_2}=w_0w_{P_{2}}$ of $w_0$ in $W^{P_2}$, where $P_2$ is the maximal parabolic subgroup with  $I^{P_2}=\{2\}$. Note that $\ell(w_0)=20$ and $\ell(w_{P_{2}})=7$, so that $\ell(w)=13$. This element is given in terms of an explicit reduced expression by
 \[\dynkin [label,label macro/.code={s_{\drlap{#1}}},
 edge length=.75cm]D5\ , \qquad
w=w^{P_2}=s_2s_3s_4s_1s_2s_3s_5s_3s_4s_2 s_3 s_1s_2.
\]
In the associated ordering of $\check R^+(w)$ the first element is $\checkalpha_2$, and the last is
\[
\checketa^{\mathbf i}_{r=13}= s_2s_1s_3s_2s_4s_3s_5s_3s_2s_1s_4s_3(\checkalpha_2)=\checkalpha_1+\checkalpha_2+2\checkalpha_3+\checkalpha_4+\checkalpha_5,
\]
where $\mathbf i=(2,3,4,1,2,3,5,3,4,2,3,1,2)$. The sum of the first and last element of $\check R^+(w)$ is still a positive coroot, and thus again in $\check R^+(w)$. Namely,
\[
\checketa^{\mathbf i}_{13}+\checkalpha_2=\checkalpha_1+2\checkalpha_2+2\checkalpha_3+\checkalpha_4+\checkalpha_5=\check\theta,
\]
the highest coroot for type $D_5$. We therefore have in $\check R^+(w)$ that
\[
\checkalpha_2\ <_{\mathbf i}\
\checketa^{\mathbf i}_{13}+\checkalpha_2\ <_{\mathbf i}\ \checketa^{\mathbf i}_{13}.
\]
But since $\checkalpha_2$ is the only simple coroot in $\check R^+(w)$ and every ordering of $\check R^+(w)$ coming from a reduced expression must start with $\checkalpha_2$, we see that there is no reduced expression of $w$ for which these  coroots appear in the opposite order.

However, the decomposable element $\check\theta$ can also be described as $\check\theta=\checketa^{\mathbf i}_{7}=s_2s_1s_3s_2s_4s_3(\checkalpha_5)$. There are two other decompositions of $\check\theta$ in $\check R^+(w)$,
\begin{eqnarray}
\check\theta &=&
\checketa^{\mathbf i}_{3}+\checketa^{\mathbf i}_{11}=\checkalpha_{23}+\checkalpha_{12345},\\
\check\theta &=&
\checketa^{\mathbf i}_{6}+\checketa^{\mathbf i}_8=\checkalpha_{1234}+\checkalpha_{235}.
\end{eqnarray}
For the second one there is a reduced expression $\mathbf j$,
\[
w=s_2s_3s_4s_1s_2\,\boxed{s_5s_3s_5}\,s_4s_2 s_3 s_1s_2,
\]
where the summands are in the reverse order.
 \end{example}

\section{The Picard group of $\XX_{w,P}$}\label{s:Picard}
In this section, we let  $w\in W^P$, and provide precise generators of the Picard group of $\XX_{w, P}$ in \cref{p:PicardGroup}.
\begin{defn}\label{d:RichardsonDivisors}
 Recall that $I=I_P\sqcup I^P$ with $I_P=\{k\in I \mid \dot s_k\in P\}$. Moreover, $I^B=I$ for the case $P=B$. Associate to $w$ the sets
\begin{eqnarray*}
 I_w^B&:=&I_w\ \,= \ \, \{i\in I\mid s_i<w\},
 \\
 I^P_w&:=&I_w\cap I^P=\{k\in I^P\mid s_k<w\},
\end{eqnarray*}
and associate to $i\in I_w^B$ the Weil divisor $\XX^{s_i}_{w,P}$ in $\XX_{w,P}$. If $k\in I_w^P$ then $\XX^{s_k}_{w,P}$ is in fact a Richardson divisor, see \cref{r:Richardson}, and we will also use the notation $D_k:=\XX^{s_k}_{w,P}$ for it, as these divisors will play a special role.
\end{defn}
The proposition below is \cref{mt:PicardGroup} from the introduction.
\begin{prop}\label{p:PicardGroup}
The divisor $\XX^{s_i}_{w,P}$ in $\XX_{w,P}$ associated to $i\in I_w$ is Cartier if $i\in I^P_w$. Writing  $D_k=\XX^{s_k}_{w,P}$ for these Cartier divisors associated to $k\in I^P_w$, we have that the Picard group $\operatorname{Pic}(\XX_{w,P})$ is freely generated by the line bundles $\mathcal O(D_k)$.
\end{prop}

\begin{proof}
We recall that restriction of line bundles
\[
\operatorname{Res}:\Pic(G/P)\to \Pic(\XX_{w,P})
\]
is a surjection, by  \cite[Proposition 6]{Mat88}.
We first show that the divisors $D_k$ for $k\in I_w^P$ are Cartier in $\XX_{w,P}$ and the line bundles $\mathcal O_{\XX_{w,P}}(D_k)$ generate the Picard group.

Recall that $P\supseteq B$ is an `upper-triangular'  parabolic subgroup of $G$.
Natural generators of the Picard group of $G/P$ can be constructed in a standard way using Borel-Weil, see \cite{Brion}. Explicitly, in our conventions, for each fundamental weight $\omega_k$ with $k\in I^P$ we have a highest weight representation $V_{\omega_k}$ and a map
\[
G/P\overset{\phi_k}\longrightarrow \mathbb P(V_{\omega_k})
\]
sending $gP$ to the line $\langle g\cdot v^+_{\omega_k}\rangle_\CC$ spanned by a highest weight vector. The line bundles $\mathcal L_{\omega_k} :=\phi_k^*(\mathcal O(1))$ generate
the Picard group $\Pic(G/P)\cong \ZZ^{I^P}$.  Namely, the hyperplane in $\mathbb P(V_{\omega_k})$ defined by the vanishing of the highest weight component pulls back to the divisor
$
\XX^{s_k,P}$,
and we see that  $\div(\mathcal L_{\omega_i})=[\XX^{s_k,P}]$. The divisor class $[\XX^{s_k,P}]$ for $k\in I^P$ corresponds  via Poincar\'e duality to the Schubert cohomology class also denoted $\sigma_{G/P}^{s_k}\in H^2(G/P,\ZZ)$. In other words, $c_1(\mathcal L_{\omega_k})=\sigma^{s_k}_{G/P}$ describes the standard isomorphism $\Pic(G/P)\cong H^2(G/P,\ZZ)$.  If we now restrict $\mathcal L_{\omega_k}$ to the Schubert variety $\XX_{w,P}$, the associated divisor class $\div(\mathcal L_{\omega_k}|_{\XX_{w,P}})$ can be computed using intersection theory in $G/P$ via
\begin{equation}\label{e:divofL}
\iota_*(\div(\mathcal L_{\omega_k}|_{\XX_{w,P}}))=[\XX_{w,P}]\cdot \div(\mathcal L_{\omega_k})=[\XX_{w,P}]\cdot [\XX^{s_k,P}]=\sum_{\checketa\in \RpwP} \langle \omega_k,\checketa\rangle [\XX_{ws_\checketa ,P}],
\end{equation}
where $\iota:\XX_{w,P}\hookrightarrow G/P$ denotes the inclusion of the Schubert variety, and the final equality is by the Chevalley formula in Schubert calculus of $G/P$ \cite{Chevalley}, see also \cite[Lemma~8.1]{FultonWoodward}.
On the other hand, we have  $D_{k}=\pi_P(\XX^{s_k}_{w,B})$, and in $G/B$,
\begin{equation}\label{e:DiviaSchubCalc}
[\XX^{s_k}_{w,B}]_{G/B}= [\XX^{s_k,B}]_{G/B}\cdot [\XX_{w,B}]_{G/B}=
\sum_{\checketa\in \RpwB} \langle \omega_k,\checketa\rangle [\XX_{ws_\checketa ,B}]_{G/B} .
\end{equation}
Pushing forward via $\pi_P:G/ B\to G/P$ and comparing \eqref{e:DiviaSchubCalc} with \eqref{e:divofL}, we find that
\[
[D_k]_{G/P}=(\pi_P)_*([\XX^{s_k}_{w,B}]_{G/B})=\iota_*(\div(\mathcal L_{\omega_k}|_{\XX_{w,P}})),
\]
where the extra summands in \eqref{e:DiviaSchubCalc} disappear in the pushforward for dimension reasons. Thus, as a divisor in $\XX_{w,P}$ we have that $D_k$ is Cartier and $\mathcal O_{\XX_{w,P}}(D_k)=\mathcal L_{\omega_k}|_{\XX_{w,P}}$.

Since the line bundles $\mathcal L_{\omega_k}$ for $k\in I^P$ generate $\Pic(G/P)$, and $\operatorname{Res}$ is a surjection, it follows that the $\operatorname{Res}(\mathcal L_{\omega_k})=\mathcal O_{\XX_{w,P}}(D_k)$, generate $\Pic(\XX_{w,P})$. Of course, if $s_k\not <w$, then the divisor class $[D_{k}]_{\XX_{w,P}}=0$ and $\mathcal L_{\omega_k}|_{\XX_{w,P}}$ is trivial. Thus the line bundles $\mathcal O_{\XX_{w,P}}(D_k)$ with $k\in I^P_w$ suffice to generate the Picard group of $\XX_{w,P}$.

We can check directly that $\Pic(\XX_{w,P})$ is freely generated by the  $\mathcal O_{\XX_{w,P}}(D_k)$, equivalently, that there are no non-trivial linear relations between the $[D_k]$ as follows. Suppose we have that $D=\sum_{k\in I^P_w} c_k D_k$ is linearly equivalent to $0$. Let $\lambda=\sum_{k\in I^P}c_k\, \omega_k$ and consider $\mathcal L_{\lambda}\in \Pic(G/P)$. Then since $\mathcal O_{\XX_{w,P}}(D)=\operatorname{Res}(\mathcal L_{\lambda})$ we have that $\mathcal L_{\lambda}|_{\XX_{w,P}}$ is a trivial line bundle.
Now for any $s_k<w$ we can restrict $\mathcal L_{\lambda}$ further to the $1$-dimensional Schubert variety $\XX_{s_k,P}\subseteq \XX_{w,P}$. This restriction is still trivial. But the line bundle $\mathcal L_{\lambda}|_{\XX_{s_k,P}}$  on the $\XX_{s_k,P}$ has degree computed by $\langle\lambda, \checkalpha_k\rangle=c_k$, as in \cite[Lemma 3.2]{FultonWoodward}. Thus we obtain that $c_k=0$ and the linear relation is trivial.
\end{proof}

\begin{remark}\label{r:PicDiv}
We have the following commutative diagram,
\[\begin{tikzcd}
	{\Pic(G/P)} & {\Pic(\XX_{w,P})} \\
	{H^2(G/P,\ZZ)} & {H^2(\XX_{w,P},\ZZ)} && {H_{2\ell(w)-2}(\XX_{w,P},\ZZ).}
	\arrow["{\operatorname{Res}}", two heads, from=1-1, to=1-2]
	\arrow["{c_1}"', from=1-1, to=2-1]
	\arrow["{c_1}"', from=1-2, to=2-2]
	\arrow["\div",   from=1-2, to=2-4]
	\arrow["{i^*}", two heads, from=2-1, to=2-2]
	\arrow["{\cap [\XX_{w,P}]}", from=2-2, to=2-4]
\end{tikzcd}\]
Here the first Chern class maps $c_1$ are both isomorphisms, using \cref{p:PicardGroup} for the right-hand one. We refer to  \cite[Section 1.1 C]{Laz} for the $c_1$ map for singular Schubert varieties.
The divisor class map $\div$ is injective, and thus so is the cap-product map
\begin{equation}\label{e:cap}
\begin{array}{ccc}
{H^2(\XX_{w,P},\ZZ)} &\xrightarrow{\cap [\XX_{w,P}]}& H_{2\ell(w)-2}(\XX_{w,P},\ZZ).\\
\end{array}
\end{equation}
We may identify $\Pic(\XX_{w,P})$ with its image in $H_{2\ell(w)-2}(\XX_{w,P},\ZZ)$, which by \cref{p:PicardGroup} is precisely the free span of the Cartier divisor classes $[D_k]$, where $k\in I^P_w$. Moreover, we have

\begin{equation}\label{e:Dihom}
[D_k]=
\sum_{\checketa\in \RpwP} \langle \omega_k,\checketa\rangle [\XX_{ws_\checketa,P}].
\end{equation}
From this perspective, \cref{p:PicardGroup} implies that the matrix $(\langle \omega_k,\checketa\rangle)_{\checketa,k}$ with $\checketa\in \RpwP$ and $k\in I^P_w$ has full rank equal to $|I^P_{w}|$. We note that this is a non-trivial statement in terms of Weyl group combinatorics. Moreover, the span $\langle [D_k]_{\XX_{w,P}}\rangle_\ZZ$ as a sublattice inside $H_{2\ell(w)-2}(\XX_{w,P},\ZZ)$ (the Cartier divisor classes inside the group of Weil divisor classes) encodes more subtle  information about the singularities of the Schubert variety $\XX_{w,P}$. We make use of this in the following section.
\end{remark}

\section{Factoriality of Schubert varieties}\label{s:factorial}
Let us begin by defining a pair of lattices which will play a central role in analysing factoriality.
Recall that $H^2(\XX_{w,P},\ZZ)$ can be identified with the Picard group of $\XX_{w,P}$, as seen above, and we have a basis given by the Chern classes $c_1(\mathcal L_{\omega_k}|\XX_{w,P})$ of the generators of $\Pic(\XX_{w,P})$. This leads us to identify $H^2(\XX_{w,P},\ZZ)$ with the sublattice of the weight lattice of $G$ spanned by the fundamental weights $\omega_k$ with $k\in I^P_w$.

\begin{defn}\label{d:QPw} Let $w\in W^P$. We define a dual pair of $\ZZ$-lattices in terms of the root datum of $G$, which are naturally identified with second cohomology and homology of $\XX_{w,P}$ as shown.
\begin{equation*}
\begin{array}{cclcc}\label{e:QwP}
\Pcal_{w,P} &:=&\langle\ \omega_{k}\mid k\in I_w^P\ \rangle_\ZZ, &\hat=&H^2(\XX_{w,P},\ZZ),\\
\Qcal_{w,P} &:= &\langle\, \checkalpha_i\mid i\in I_w\, \rangle_\ZZ\ /\ \langle \checkalpha_j\mid j\in I_P\cap I_w\rangle_\ZZ &\hat=& H_2(\XX_{w,P},\ZZ).
\end{array}
\end{equation*}
Here $\Pcal_{w,P}$ is a sublattice of the weight lattice of $G$, and $\Qcal_{w,P}$ is a subquotient of the coroot lattice of $G$. The dual pairing is inherited from the dual pairing between the weight and coroot  lattices. We also note that $\RpwP$ naturally maps to $\Qcal_{w,P}$, and we do not distinguish notationally between an element or $\RpwP$ and its image in $\Qcal_{w,P}$.
\end{defn}
We can now state the main results of this section. The following proposition relates properties of the subset $\RpwP$ of $\check R^+(w)$ to the factoriality of $\XX_{w,P}$, part (3) of which is \textbf{\cref{mt:factorialgen}}.

\begin{prop} \label{p:Qcal}
Let $w\in W^P$. The associated subset $\RpwP$  of $w$ 0has the following properties.
\begin{enumerate}
\item The $|\RpwP|\times |I^P_{w}|$ matrix $(\langle\omega_i,\checketa\rangle)_{\checketa, k}$, has rank $|I^P_w|$, and the set $\RpwP$ generates a sublattice of full rank inside $\Qcal_{w,P}$. 
\item The Schubert variety $\XX_{w,P}$ is $\mathbb Q$-factorial if and only if the set $\RpwP$ has cardinality $|I^P_w|$. Equivalently, this is if $\RpwP$ defines a $\QQ$-basis of  $\Qcal_{w,P}\otimes_\ZZ\QQ$.
\item The Schubert variety $\XX_{w,P}$ is factorial if and only if $\RpwP$ is a $\ZZ$-basis of $\Qcal_{w,P}$.
\end{enumerate}
\end{prop}

Then we will prove the theorem below concerning the simply-laced case.

\begin{theorem}\label{t:Factorialsl}
Suppose $G$ is of simply-laced type and $w\in W^P$. Then we have the following.
\begin{enumerate}
\item
$\RpwP$ generates $\Qcal_{w,P}$.
\item
If $D$ is a Weil divisor in $\XX_{w,P}$ and $D$ is $\QQ$-Cartier, then $D$ is in fact Cartier.
\item
$\XX_{w,P}$ is factorial if and only if $|\RpwP|=|I^P_w|$.
\end{enumerate}
\end{theorem}

\begin{remark}\label{r:Emoto}
In type $A$,
Bousquet-M\'elou and Butler \cite{BMB07} characterised the permutations giving rise to factorial Schubert varieties in $SL_{n}/B$ in combinatorial terms and via a pattern-avoidance criterion, proving a conjecture of  Woo and Yong from \cite{WooYong}. In \cite[Proposition A.6]{Emoto} this criterion was shown to be equivalent to $|\RpwB|=|I_w^B|$, which proves \cref{t:Factorialsl}(3) for Schubert varieties in type $A$ full flag varieties. 
In \cite[Conjecture A.8]{Emoto} it was conjectured for general type Schubert varieties in $G/B$  that being factorial is equivalent to $|\RpwB|=|I^B_w|$. \cref{t:Factorialsl} implies that this conjecture holds when $G$ is simply-laced (and even extends to $G/P$  Schubert varieties). We observe that the conjecture is false in non-simply-laced type, though, see \cref{ex:G2} and \cref{r:G2}. Namely, the condition $|\RpwB|=|I^B_w|$ is equivalent to $\XX_{w,B}$ being $\QQ$-factorial, and there are $\QQ$-factorial Schubert varieties that are not factorial in non-simply-laced types.
\end{remark}

\begin{remark}\label{r:G2}
Applying \cref{p:Qcal} to \cref{ex:G2} we see that all Schubert varieties for $G_2$ are $\mathbb Q$-factorial, since $\RpwP$ is always a linearly independent set of cardinality $|I^P_w|$ in the case of $G_2$. Amongst these, the following list of Schubert varieties are not factorial, again as an application of the proposition:
\[
\XX_{s_2s_1s_2,P_2},\quad \XX_{s_2s_1s_2,B},\quad \XX_{s_1s_2s_1s_2,B}, \quad\XX_{s_2s_1s_2s_1s_2,B},\]
as well as
\[
\XX_{s_2s_1, P_1},\quad \XX_{s_1s_2s_1, P_1},\quad \XX_{s_2s_1s_2s_1, P_1},\quad\XX_{s_2s_1s_2s_1, B}.
\]
Namely, for each of these $\XX_{w,P}$ the $\ZZ$-span of $\RpwP$ is a proper sublattice of $\Qcal_{w,P}$, giving rise to some Weil divisors that are not Cartier. This can be seen directly by inspection of the sets $\RpwP$ which are described in \cref{ex:G2}.

We may compare the $G/B$ Schubert varieties listed above with the list of singular Schubert varieties for $G_2$ given in \cite[Theorem~2.4]{BilleyPostnikov}. In that list there is one additional singular Schubert variety, $\XX_{s_1s_2s_1s_2s_1,B}$. Since $R^+_{s_1s_2s_1s_2s_1, B}=\{\checkalpha_1,\checkalpha_1+\checkalpha_2\}$ this Schubert variety appears to be factorial and in that sense `less singular' than the others.
\end{remark}

 \cref{p:Qcal} is a consequence of the proof of \cref{p:PicardGroup} together with known properties of Schubert varieties, and we can give a proof straight away.

\begin{proof}[{Proof of \cref{p:Qcal}}]
The property (1) is a consequence of \cref{p:PicardGroup}, see also \cref{r:PicDiv}.
Now recall that $\RpwP$ is in bijection with the Schubert divisors of $\XX_{w,P}$. The Schubert variety $\XX_{w,P}$ is normal and has rational singularities \cite{Andersen}. Therefore it is $\QQ$-factorial if and only if the Betti numbers $b_2=\dim(H^2(\XX_{w,P},\QQ))$ and $b_{2\ell(w)-2}=\dim(H_{2\ell(w)-2}(\XX_{w,P},\QQ))$ agree, by \cite[Theorem~A]{ParkPopa}. Thus $\XX_{w,P}$ is $\QQ$-factorial if and only if the order of $\RpwP$ is equal to $b_2$, which is also $|I^P_w|$ and the rank of $\Qcal_{w,P}$. Since $\RpwP$ generates a full rank sublattice of $\Qcal_{w,P}$ by (1) this is furthermore equivalent to $\RpwP$ being a $\ZZ$-basis of that sublattice, and after tensoring with $\QQ$, a basis of $\Qcal_{w,P}\otimes\QQ$. Therefore (2) is proved. If $\RpwP$ is a $\ZZ$-basis of $\Qcal_{w,P}$ then the matrix from (1) is invertible (over $\ZZ$) and by \eqref{e:Dihom} we have that the Cartier divisor classes $[D_i]$ form a basis of $H_{2\ell(w)-2}(\XX_{w,P},\ZZ)$, which implies that $\XX_{w,P}$ is factorial and vice versa.  Part~(3) of the proposition follows.
\end{proof}

The proof of \cref{t:Factorialsl} will take up the remainder of the section. It involves constructing in simply-laced type a subset of Schubert classes in $H^{2\ell(w)-2}(\XX_{w,P},\ZZ)$ that map to a basis of $H_2(\XX_{w,P},\ZZ)$ under the dual of the cap product map \eqref{e:cap}, as explained in \cref{r:dualcap}.

This concrete construction will in particular imply that the cap product map has saturated image, that is, the Cartier divisor classes form a saturated sublattice in $H_{2\ell(w)-2}(\XX_{w,P},\ZZ)$ in simply-laced types.

\subsection{Construction of $\mathcal B_{w,B}$}\label{s:BwB} We temporarily restrict to the case $P=B$ in this subsection. Our goal is to construct a special subset of $\RpwB$ of cardinality $|I^B_w|$ in the simply-laced setting.

\begin{prop}\label{p:Bw} Let $G$ be of simply-laced type and  $w\in W$. There exists a subset
\[
\mathcal B_{w,B}:=\{\checketa_w(k)\mid k\in I_w^B\}
\]
of $\RpwB$ and an ordering of $I^B_w$ such that the matrix $M=(\langle\omega_i,\checketa_w(k)\rangle)_{k,i\in I^B_w}$ is unipotent lower-triangular. In particular, $\mathcal B_{w,B}$ embeds into $\Qcal_{w,B}$ as a $\ZZ$-basis.
\end{prop}

\begin{remark}
Whenever $w$ has full support, \cref{p:Bw} constructs a basis of the coroot lattice. In the example of $w=w_0$, this basis is just the basis of simple coroots.
\end{remark}

In preparation for the proof of this proposition we make the following definition.

\begin{defn}\label{def: etawk}
Let $w \in W$ and $s_k<w$.
For a reduced expression $w=s_{i_1}s_{i_2}\dotsc s_{i_r}$, we  write  $\mathbf i=(i_1,\dotsc, i_r)$, and define
$d_{\bf i}(k)$ to be  the distance of the right-most occurrence of $s_k$ in the  reduced expression $\mathbf i$ from the end of $\mathbf i$. In other words, $\ell=r-d_{\bf i}(k)+1$ is the right-most position of $k$ in $\mathbf i$, meaning $i_{\ell}=k$ and $i_m\ne k$ whenever $m>\ell$.
Now we define the `absolute minimal distance' from the end,
\begin{align}
d_w(k):=\min(\{d_{\bf i}(k)\mid \mathbf i\text{ is a reduced expression of } w\}).
\end{align}
If $d_{\bf i}(k)=d_w(k)$ we say that the reduced expression $\mathbf i$ is `rightmost' for $k$.
Let $\ell=r+1-d_w(k)$ so that in $i_\ell=k$ is the rightmost occurrence of $k$ in $\mathbf i$. We define an element $\checkalpha_w^{\mathbf i}(k)\in R^+(w)$ by
\begin{equation}\label{e:etawk}
\checkalpha_w^{\mathbf{i}}(k):=s_{i_r}s_{i_{r-1}}\dotsc s_{i_{\ell+1}}(\checkalpha_k).
\end{equation}
\end{defn}
 We often simply write
$\checkalpha_w(k):=\checkalpha_w^{\mathbf{i}}(k)$, namely  one  reduced expression  $\mathbf i$ with $d_{\mathbf i}(k)=d_w(k)$ is a prior fixed.
The coroot $\checkalpha_w(k)$
generally  depends on the choice of $\mathbf i$, however (see \cref{ex:A4nonunique}). When $\mathbf i$ is omitted from the notation, this will indicate  that the choice of $\mathbf i$ is not relevant to our purpose.

\begin{proof}[Proof of \cref{p:Bw}]  For every $k\in I^B_w$ let $\checkalpha_w(k)$ be as constructed in \cref{def: etawk}. Namely, we have
\begin{equation}\label{e:rightmostalpha}
\checkalpha_w(k)=s_{i_r}\dotsc s_{i_{\ell+1}}(\checkalpha_k)
\end{equation}
in terms of a right-most reduced expression for $k$.
Note that we can write $\checkalpha_w(k)$ as
\begin{equation}\label{e:alphawkexpanded}
\checkalpha_w(k)=\checkalpha_{k}+\sum_{m=\ell+1}^r c_m\checkalpha_{i_m},
\end{equation}
for some coefficients $c_m\in \ZZ_{\ge 0}$. Here the coefficient of $\checkalpha_k$ is $1$ using the simply-laced assumption and the fact that none of the simple reflections occurring in \eqref{e:rightmostalpha} include $s_{k}$. So we have $\langle \omega_k,\checkalpha_w(k)\rangle=1$. We now define $\checketa_w(k)$ as follows.

If $\checkalpha_w(k)\in \RpwB$, then we set $\checketa_w(k):=\checkalpha_w(k)$. If not then, by the simply-laced assumption and \cref{p:indecomposableV2} and \cref{l:decomposablesl}, we have that
$\checkalpha_w(k)$ is decomposable into a sum of elements of $\RpwB$. Now precisely one of these indecomposable summands $\checkmu$ has the property $\langle\omega_k,\checkmu\rangle=1$, and so we set $\checketa_w(k)=\checkmu$. Note that by construction of  $\checketa_w(k)$  as a summand of $\checkalpha_w(k)$ we obtain from \eqref{e:alphawkexpanded} an expansion for $\checketa_w(k)$,
\begin{equation}\label{e:etawkexpanded}
\checketa_w(k)=\checkalpha_{k}+\sum_{m=\ell+1}^r c'_m\checkalpha_{i_m}
\end{equation}
with $c'_m\in\ZZ_{\ge 0}$.

We now choose a total ordering  $\triangleleft$ on $I^B_w$,  such that $d_w(k_1)< d_w(k_2)$ implies that $k_1\triangleleft k_2$.  Consider the matrix $M=(\langle\omega_i,\checketa_w(k)\rangle)_{k,i\in I^B_w}$, with rows and columns ordered according to $(I^B_w,\triangleleft)$.
Note that any $i=i_m$ occurring in the sum in \eqref{e:etawkexpanded} has $d_w(i)<d_w(k)$. Thus any simple coroot $\checkalpha_i$ that appears in $\checketa_w(k)$ has $d_w(i)<d_w(k)$, and therefore $i\triangleleft k$ in the ordering of $I^B_w$. Thus, $M$ is lower-triangular. Moreover, as seen in \eqref{e:etawkexpanded}, $\langle\omega_k,\checketa_w(k)\rangle=1$, implying that  $M$ has $1$'s along the diagonal.
\end{proof}

\begin{remark}\label{r:dualcap} Recalling that $\Qcal_{w,B}= H_2(\XX_{w,B},\ZZ)$, the basis constructed in the lemma can be considered as giving rise to a (nonstandard) basis of the second homology group of $\XX_{w,B}$. Namely there is a curve class $[C_\checketa]$  associated to $\checketa\in \mathcal B_{w,B}$, which has the property that for every line bundle $\mathcal L_{\omega_i}|_{\XX_{w,B}}=\mathcal O(D_i)$ we have
\begin{equation}\label{e:dualcap}
\langle c_1(\mathcal O(D_i)),[C_\checketa]\rangle\overset{\small{\circled{1}}}{=\joinrel=}\langle\omega_i,\checketa\rangle\overset{\small{\circled{2}}}{=\joinrel=}\langle {[D_i]},\sigma^{w s_\checketa}\rangle,
\end{equation}
where $\sigma^{w s_\checketa}$ is the Schubert cohomology class in $H^{2\ell(w)-2}(\XX_{w,B},\ZZ)$ associated to $ws_\checketa$ and $i\in I^P_w$. Here $\circled{1}$ can be taken as the definition of $[C_\checketa]$ and $\circled{2}$  follows from  \eqref{e:Dihom}. This identity expresses the fact that the curve class $[C_\checketa]$ is the image of the Schubert class $\sigma^{w s_\checketa}$ under the dual of the cap product map \eqref{e:cap}. In other words, if we think of $\RpwB$ as indexing the Schubert basis of $H^{2\ell(w)-2}(\XX_{w,B},\ZZ)$, then $\mathcal B_{w,B}$ indexes a subset of this Schubert basis for which the associated curve classes $[C_\checketa]$ form a basis of $H_2(\XX_{w,B},\ZZ)$. In the smooth case, the elements of $\mathcal B_{w,B}$ are of course all the Schubert basis elements $\sigma^w$ in $H^{2\ell(w)-2}(\XX_{w,B},\ZZ)$, and the $[C_\checketa]$ are their Poincar\'e dual curve classes. In the general  non-simply-laced case such a basis of $H_{2}(\XX_{w,B},\ZZ)$ may not exist.
\end{remark}
\begin{remark}
 {We expect that} a weaker version of \cref{p:Bw} with `unipotent lower-triangular' replaced by `lower-triangular with diagonal entries in $\ZZ_{>0}$' holds in the non-simply-laced case. In the $G_2$ case this can be seen directly, see \cref{ex:G2}.
\end{remark}

Now we illustrate the construction of $\mathcal B_{w,B}$ in some examples.

\begin{example}\label{ex:A4nonunique} Consider type $A_4$. Consider the  following two   reduced expressions, $\mathbf{i}$ and $ \mathbf{i}'$, of the permutation $53142$ in $S_5$,
\begin{equation}\label{e:wred}
w=s_2s_1s_3s_4s_3s_2s_1=s_4s_3s_2s_1s_3s_2s_4.
\end{equation}
We simply denote  $\checkalpha_{j_1}+\checkalpha_{j_2}+\cdots+\checkalpha_{j_t}$ as  $\checkalpha_{j_1j_2\ldots j_t}$. Associated to $w$ we have
\[
\RpwB=\{\checkalpha_1,\checkalpha_2,\checkalpha_4,\checkalpha_{12},\checkalpha_{123},\checkalpha_{234}\}\qquad \text{and}\qquad \check R^+(w)=\RpwB\cup\{\checkalpha_{1234}\}.
\]
For $k=1,2,4$ clearly $d_w(k)=1$ and $\checkalpha_w(k)=\checkalpha_k$. But $d_w(3)=3$ and there are two options for $\checkalpha_w(3)$, depending on the choice of reduced expression $\mathbf i$. Namely, $\checkalpha^{\mathbf i}_w(3)=\checkalpha_{123}$, using the first reduced expression in \eqref{e:wred}, and $\checkalpha^{\mathbf i'}_w(3)=\checkalpha_{234}$, using the second one. Either of the two reduced expressions  gives rise to a subset with expected property in \cref{p:Bw}, namely
\[\mathcal B^{\mathbf i}_{w,B}=\{\checkalpha_1,\checkalpha_2,\checkalpha_4,\checkalpha_1+\checkalpha_2+\checkalpha_3\},\qquad \mathcal B^{\mathbf i'}_{w,B}=\{\checkalpha_1,\checkalpha_2,\checkalpha_4,\checkalpha_2+\checkalpha_3+\checkalpha_4\}.
\]
Note however, that the dual bases
\[\mathcal B^{\mathbf i.*}_{w,B}=\{\omega_1-\omega_3,\omega_2-\omega_3,\omega_4,\omega_3\},\qquad \mathcal B^{\mathbf i',*}_{w,B}=\{\omega_1,\omega_2-\omega_3,\omega_4-\omega_3,\omega_3\},
\]
while they differ, have the same sums: $\omega_1+\omega_2-\omega_3+\omega_4$.  This reflects the fact that this Schubert variety is Gorenstein, see \cref{r:PisB}.
\end{example}

\begin{example}\label{ex:D5counterex}
Consider in type $D_5$ the following element  $w\in W^{P_3}$, where $P_3$ is the maximal parabolic subgroup with   $I^{P_3}=\{3\}$. We write down three  reduced expressions   $\mathbf{i}$, $ \mathbf{i}'$ and $\mathbf{i}''$ of $w$ as follows.
\[\dynkin [label,label macro/.code={s_{\drlap{#1}}},
 edge length=.75cm]D5\ , \qquad
w=s_2s_3s_1s_2s_3s_4s_5s_3=s_2s_3s_1s_2s_3s_5s_4s_3 =s_1s_2s_3s_5s_4s_1s_2s_3.
\]
  We have
\[
\check R^+(w)=\{\checkalpha_3,\checkalpha_{35},\checkalpha_{34},\checkalpha_{345},\checkalpha_{2345}+\checkalpha_3,\ \checkalpha_{12345}+\checkalpha_2,\ \checkalpha_{23},\checkalpha_{123}\},
\]
and the subset of indecomposable elements is.
\[
\RpwB=\{\checkalpha_3,\checkalpha_{35},\checkalpha_{34},\checkalpha_{345},\checkalpha_{23},\checkalpha_{123}\}.
\]
We can find rightmost elements $\checkalpha_w(k)$  for every index $k$ using $\mathbf{i}$, $\mathbf{i}'$ and $\mathbf{i}''$, and these elements are indecomposable. We therefore have
\begin{eqnarray*}
\checketa_w(3)=\checkalpha_w^{\mathbf{i}}(3)&=&\checkalpha_3,\\
\checketa_w(5)=\checkalpha_w^{\mathbf{i}}(5)&=&\checkalpha_3+\checkalpha_5,\\
\checketa_w(4)=\checkalpha_w^{\mathbf{i}'}(4)&=&\checkalpha_3+\checkalpha_4,\\
\, \checketa_w(2)=\checkalpha_w^{\mathbf{i}''}(2)&=&\checkalpha_{2}+\checkalpha_{3},\\
\checketa_w(1)=\checkalpha_w^{\mathbf{i}''}(1)&=&\checkalpha_{1}+\checkalpha_2+\checkalpha_{3}.
\end{eqnarray*}
Then  $\mathcal B_{w,B}=\{\checketa_w(k)\mid k=1,\dotsc, 5\}$ is a subset of $\RpwB$ that descends to a basis of  $\Qcal_{w,B}$.
\end{example}

\begin{conj}
Let $G$ be of simply-laced type and $w\in W$. For $k\in I_w^B$ consider a reduced expression $\mathbf{i}$ of $w$ such that $d_{\mathbf{i}}(k)=d_w(k)$ and  let $\checketa_w(k)=\checkalpha_w^{\mathbf{i}}(k)$. Then $\checketa_w(k)$ is indecomposable in $\check R^+(w)$. In  particular, $\{\checketa_w(k)\mid k\in I^B_w\}$ constructed in this way satisfies the expected properties from \cref{p:Bw}.
\end{conj}

We note that the simply-laced assumption is necessary for this conjecture.

\subsection{Construction of $\mathcal B_{w,P}$}\label{s:BwP} The generalisation of the basis $\mathcal B_{w,B}$ to the parabolic setting is not completely straightforward. It   will use a recursive procedure that is based on \cref{l:charWPcond}. Namely \cref{l:M} below provides us with a mechanism for modifying the basis $\mathcal B_{w,B}$ in order to ultimately construct our basis of $\RpwP$. In the following we fix $\mathcal B_{w,B}=\{\checketa_w(k)\mid k\in I^B_w\}$ constructed as in \cref{p:Bw}.
\begin{defn}\label{def: Padaptation}  Let $G$ be of simply-laced type.
We call a subset $\mathcal M\subseteq \RpwB$ a $P$-adaptation of $\mathcal B_{w,B}$, if there is a bijection $\checkmu=\check\mu_w:I^P_w\to\mathcal M$ with the property that
\[\checkmu(k)\equiv \, \checketa_w(k)\mod \langle\checkalpha_j\mid j\in I_P\rangle\qquad \text{for all $k\in I^P_w$}.
\]
\end{defn}
\begin{lemma}\label{l:M} Let $G$ be of simply-laced type.
Let $\mathcal M$ be a $P$-adaptation of $\mathcal B_{w,B}$. Suppose that there exist   $k_0\in I^P_w$ and   $j\in I_P$ such that $ws_{\checkmu(k_0)}(\checkalpha_j)<0$.
Then  $\checkmu'(k_0):=s_{\checkmu(k_0)}(\checkalpha_j)$ satisfies:
\begin{enumerate}
\item $\checkmu'(k_0)\in \RpwB$;
\item $\checkmu'(k_0)\equiv \checkmu(k_0)\mod \langle\checkalpha_j\mid j\in I_P\rangle_\ZZ$; 
\item the heights are related by $\height(\checkmu'(k_0))> \height(\checkmu(k_0))$.
\end{enumerate}
For $k\ne k_0$ let us set $\checkmu'(k)=\checkmu(k)$, then we have   another $P$-adaptation $\mathcal M'$ of $\mathcal B_{w,B}$ defined by
\[
\mathcal M'=\{\checkmu'(k)\mid k\in I^P_w\}.
\]
\end{lemma}

\begin{proof}
By  \cref{l:charWPcond} we have $\checkmu'(k_0)\in \check R^+(w)$. On the other hand we have
\begin{equation}\label{e:mu'k0}
\checkmu'(k_0)=s_{\checkmu(k_0)}(\checkalpha_j)=\checkalpha_j-\langle\checkalpha_j,\checkmu(k_0)^\vee\rangle\checkmu(k_0).
\end{equation}
Note that $w(\checkmu'(k_0))<0$ implies $\checkmu'(k_0)\ne\checkalpha_j$ (indeed, $w(\checkalpha_j)>0$ as $w\in W^P$). Therefore \eqref{e:mu'k0} together with the fact that $\checkmu'(k_0)>0$ implies that we must have $c':=-\langle\checkalpha_j,\checkmu(k_0)^\vee\rangle>0$. Since $G$ is of simply-laced type, the only possible positive factor $c'$ is $1$. Therefore we are simply replacing $\checkmu(k_0)$ by $\checkmu'(k_0)=\checkmu(k_0)+\checkalpha_j$. This implies (2). To prove (1) it remains to show that $\checkmu'(k_0)$ is indecomposable, by \cref{p:indecomposableV2}. Namely, suppose indirectly it has a decomposition  $\checkmu'(k_0)=\checkmu(k_0)+\checkalpha_j=\checkmu_1+\checkmu_2$.   Then $s_j(\checkmu'(k_0))=\checkmu(k_0)=\checkmu'(k_0)-\alpha_j$ and so we must have that one of the summands, without loss of generality $\checkmu_1$, satisfies $s_j(\checkmu_1)=\checkmu_1-\alpha_j$ (and the other $s_j(\checkmu_2)=\checkmu_2$). Therefore, $\checkmu_1=\checkmu_1'+\checkalpha_j$ for another positive root $\checkmu_1'$.   Now $\checkmu_1\in \check R^+(w)$ implies $w(\checkmu_1'+\checkalpha_j)=w(\checkmu_1')+w(\checkalpha_j)<0$. But $w(\checkalpha_j)>0$ since $j\in I_P$, so that $w(\checkmu_1')<0$ necessarily. Then $\checkmu_1'+\checkmu_2=\checkmu(k_0)$ gives a decomposition of $\checkmu(k_0)$ with summands in $\check R^+(w)$ contradicting the assumption that $\checkmu(k_0)$ is indecomposable. Thus we have proved (1).

It is clear that $\height(\checkmu'(k_0))> \height(\checkmu(k_0))$ and (3) is also proved. Finally, $\mathcal M'$ is again a $P$-adaptation of $\mathcal B_{w,B}$ by (1) and (2).
\end{proof}
\begin{defn}[{Recursive construction of $\mathcal B_{w,P}$}]\label{d:BwP} Suppose $G$ is of simply-laced type and $w\in W^P$. Recall
$\mathcal B_{w,B}=\{\checketa_w(k)\mid k\in I_w^B\}$ was constructed in \cref{p:Bw}. We   construct a subset $\mathcal B_{w,P}$ of $\check R^+(w)$ as follows. Start with the subset $\mathcal M_0:=\{\checketa_w(k)\mid k\in I_w^P\}$ of $\check R^+(w)$, which is clearly a $P$-adaption of $\mathcal B_{w,B}$. Suppose $d\in\ZZ_{\ge 0}$ and we have already constructed a $P$-adaption $\mathcal M_{d}=\{\checkmu(k)\mid k\in I^P_w\}$ of $\mathcal B_{w,B}$. If for every $\checkmu\in \mathcal M_d$ we have that $ws_\checkmu\in W^P$ then set $\mathcal B_{w,P}:=\mathcal M_d$. Otherwise, pick a $k_0\in I^P_w$ and $j\in I_P$ such that $w s_{\checkmu(k_0)}(\checkalpha_j)<0$ and define
\[
\mathcal M_{d+1}:=(\mathcal M_{d})'
\]
using the construction from \cref{l:M}. This process terminates after finitely many steps since each step increases the height of one element, but the height function is bounded. 
\end{defn}

\begin{lemma}\label{p:BwP}
Suppose $G$ is of simply-laced type and $w\in W^P$. The set $\mathcal B_{w,P}$ constructed in \cref{d:BwP} has the following properties.
\begin{enumerate}
\item $\mathcal B_{w,P}\subseteq \RpwP$.
\item $\mathcal B_{w,P}=\{\checkmu(k)\mid k\in I^P_w\}$ with $\checkmu(k)\equiv \checketa_w(k)\mod \langle\checkalpha_j\mid j\in I_P\rangle_\ZZ$.
\item There exists a total ordering on $I^P_w$ for which
\[
\checkmu(k)\equiv\checkalpha_k+\sum_{k'<k}c_{k,k'}\checkalpha_{k'}\mod \langle\checkalpha_j\mid j\in I_P\rangle_\ZZ,
\]
with all coefficients in $\ZZ_{\ge 0}$.
\end{enumerate}
\end{lemma}
\begin{proof}
For property (1) the fact that the $\checkmu_{w}(k)$ lie in $\RpwB$ follows from the recursive construction and \cref{l:M}. Additionally, the recursion defining $\mathcal B_{w,P}$ only terminates if all elements $w s_{\checkmu(k)}$ lie in $W^P$. Therefore $\mathcal B_{w,P}\subseteq \RpwP$. Next, the recursion  does not alter the elements modulo $\langle \checkalpha_j\mid j\in I_P\rangle$. Therefore (2) and (3) hold as a consequence of \cref{p:Bw}.
\end{proof}

\begin{remark}\label{r:dualcapP}  The construction of $\mathcal B_{w,P}$ from \cref{d:BwP} is highly dependent on the choices. However, the image of $\mathcal B_{w,P}$ in $H_2(\checkXX_{w,P},\ZZ)$ only depends on $\mathcal B_{w,P}$ modulo $\langle\checkalpha_j\mid j\in I_P\rangle_\ZZ$ and is therefore completely determined by the choice of $\mathcal B_{w,B}$. The interpretation of $\mathcal B_{w,B}$ given in \cref{r:dualcap} extends verbatim to the simply-laced $G/P$ setting 
just with $P$ replacing $B$ everywhere. Moreover, the basis of $H_2(\checkXX_{w,P},\ZZ)$ given by $\mathcal B_{w,P}$ can be constructed directly as an image of the subset $\{\checketa_w(k)\mid k\in I^P_w\}$ of $\mathcal B_{w,B}$ 
via the map
\[
H_2(\checkXX_{w,B},\ZZ)\longrightarrow H_2(\checkXX_{w,P},\ZZ),
\]
as a consequence of \cref{p:BwP}(2).
 \end{remark}

We are now in a position to prove \cref{t:Factorialsl}.

\subsection{Proof of \cref{t:Factorialsl} and \cref{mt:simplylaced-factorial}}\label{s:Qcal}
By \cref{p:BwP} there exists a subset $\mathcal B_{w,P}$, which is related to the standard basis $\{\checkalpha_k\mid k\in I^P_w\}$ of $\Qcal_{w,P}$ by a lower-triangular integer matrix with diagonal entries $1$. It follows that $\mathcal B_{w,P}$ is itself a basis of  $\Qcal_{w,P}$, and therefore that $\RpwP$ generates the lattice $\Qcal_{w,P}$. This proves \cref{t:Factorialsl}(1). Now, for a Weil divisor $D$  we have $[D]=\sum_{\checketa\in \RpwP} n_\checketa[\XX_{ws_{\checketa},P}]$ with coefficients $n_\checketa\in \ZZ$.
If $D$ is also $\QQ$-Cartier, we may write $[D]=\sum m_k[D_k]$, a $\QQ$-linear combination of the generators $[D_k]=\sum_{\checketa\in \RpwP}\langle\omega_k,\checketa\rangle [\XX_{ws_{\checketa},P}]$ of the Cartier class group from \cref{p:PicardGroup}. Note also that the $m_k$ are unique since the $[D_k]$ are linearly independent. Or equivalently, the vectors $(\langle\omega_k,\checketa\rangle)_{\checketa\in \RpwP}$ are linearly independent, compare \cref{p:Qcal}(1). Thus we have the vector identity $(n_\checketa)_{\checketa}=\sum_{k\in I^P_w} m_k(\langle\omega_k,\checketa\rangle)_{\checketa}$.
Projecting onto the subspace with coordinates indexed by $\mathcal B_{w,P}$ we obtain $(n_{\checkmu_w(i)})_{i\in I^P_w}=\sum_{k\in I^P_w} m_k(\langle\omega_k,\checkmu_w(i)\rangle)_{i\in I^P_w}$.
But the matrix $(\langle\omega_k,\checkmu_w(i)\rangle)_{i,k\in I^P_w}$ is invertible over $\ZZ$ by \cref{p:BwP}(3). Therefore it follows that the $m_k$ all lie in $\ZZ$. In other words, $D$ is Cartier and we have proved (2). Finally, (3) follows from (1) combined with \cref{p:Qcal}(3).

We have now finished proving \cref{t:Factorialsl}.
We finally observe that \cref{t:Factorialsl} implies  \cref{mt:simplylaced-factorial} from the introduction using that $\RpwP$ indexes the Schubert divisors, so that $|\RpwP|=b_{2\ell(w)-2}(\XX_{w,P})$, and that  $|I^P_w|=b_2(\XX_{w,P})$, compare \cref{p:PicardGroup}.
\section{Gorenstein and Fano Schubert varieties}\label{s:Fano}
We now use the constructions from the previous section to study the anticanonical divisors of our Schubert varieties $\XX_{w,P}$. As a starting point we note that we can make the following explicit choice of an anticanonical divisor. The following proposition was implicit in \cite{ramanathan, KLS}, and is explained as \cite[Proposition 2.2 and Lemma 3.8]{LSZ} with a replacement of $w_0w_P$ by $w$. Here we include the arguments     adapted therein verbatim for completeness.
\begin{prop}\label{p:BasicKX} For a general Schubert variety $\XX_{w,P}$ an anticanonical divisor is given by
\[
-K_{\XX_{w,P}}=\sum_{j\in I^B_w} \XX^{s_j}_{w,P} + \sum_{\checketa\in \RpwP}\XX_{ws_\checketa,P},
\]
in terms of Schubert and projected Richardson varieties.
\end{prop}
\begin{proof} By \cite[Lemma 5.4]{KLS},  the sum of all projected Richardson hypersurface in $\XX_{w, P}=\XX_{w, P}^{\rm id}$ gives an anticanonical divisor of $\XX_{w, P}$. Hence, by \cite[Corollary 3.2]{KLS},
\[
-K_{\XX_{w,P}}=\sum_{s_j\leq_P w} \XX^{s_j}_{w,P} + \sum_{{\rm id}\leq_P x \prec  w}\XX_{x,P}^{\rm id}.
\]
Here $\leq_P$ denotes the $P$-Bruhat order, namely $u\leq_P v$ if and only if there is a chain
$u=u_1 \lessdot u_2\lessdot\cdots \lessdot u_m=v$ of Bruhat covers such that
$\pi(u_1)<\pi(u_2)<\cdots <\pi(u_m)$, where $\pi(u_m)\in W^P$ denotes the minimal length representative of the coset $u_m W_P$.  In the second sum, we notice $\XX_{x,P}^{\rm id}=\XX_{x,P}=\XX_{ws_\checketa,P}$ for $\checketa\in \RpwP$. In the first sum, by definition,  $s_j\leq_P w$ holds  only if $s_j\leq w$, namely $j\in I_w^B$.

  It remains to verify that $\XX^{s_j}_{w,P}$ does indeed have codimension $1$ in $\XX_{w,P}$. If $j\in I^P_w$ this is clear by \cref{r:Richardson}.
Now we assume $j\in I_w^B {\setminus I^P_w}$, and take a reduced expression $w=s_{i_1}\cdots s_{i_r}$ with $s_j$ appearing at the rightmost at $\ell$-th position. As in the proof of \cite[Lemma 3.8]{LSZ}, we set $u_1:=s_j$ and
$$u_m:=\begin{cases}s_{i_{r-m+1}}\cdots s_{i_{r-1}}s_{i_r}, &\mbox{if }r-\ell+1\leq m\leq r,\\
s_{i_\ell}s_{i_{r-m+2}}\cdots s_{i_{r-1}}s_{i_r},  &\mbox{if } 2\leq m\leq r-\ell.\end{cases}$$
By \cite[Lemma 3.6 (1)]{LSZ}, we have $u_m\in W^P$ for $r-\ell+1\leq m\leq r$, and $s_{i_\ell}u_m\in W^P$ for $2\leq m\leq r-\ell$.

Let $2\leq m \leq r-\ell$.
If $s_{i_\ell}u_m\neq u_m s_{i_\ell}$, then we have $u_m\in W^P$ by \cite[Lemma 3.6 (2)]{LSZ}. If the inequalities hold for all  such $m$, then we have $\pi(u_m)=u_m$, and notice $\pi(u_1)=\pi(s_j)={\rm id}$. Then $s_j\leq_P w$ by definition. Otherwise, we let $a$ denote the maximal index in $\{2, \cdots, r-\ell\}$ with the equality $s_{i_\ell}u_a= u_a s_{i_\ell}$. Let $v_k=s_{i_\ell}u_ks_{i_\ell}$ for $1\leq k\leq a$.  We have  $s_j=v_1\lessdot v_2\lessdot \cdots \lessdot v_a \lessdot u_{a+1}\lessdot \cdots \lessdot u_r$, and their projections give
${\rm id}\lessdot v_2s_{i_\ell}\lessdot \cdots \lessdot v_a s_{i_\ell} < u_{a+1}\lessdot u_{a+2}\lessdot \cdots \lessdot u_r$. Hence, we still have $s_j\leq_P w$.
\end{proof}
Now we characterise which Schubert varieties in $G/P$ are Gorenstein in the simply-laced setting, and which of those are Fano. We also describe the anticanonical line bundle for $\mathbb{Q}$-factorial Schubert varieties of arbitrary type, and  determine when it is ample. For Schubert varieties in $SL_n/B$, explicit combinatorial conditions on $w\in S_n$ characterising when $\XX_{s,B}$ is Gorenstein were given in \cite{WooYong}. The problem to find a characterisation in more general types was posed in \cite{WooYong} and again in \cite{WooYong23}.
\begin{defn}\label{d:dualbasis}
Assume $G$ is simply-laced, and $\mathcal B_{w,P}=\{\checkmu_w(k)\mid k\in I^P_w\}$ is the subset of $\RpwP$  constructed above. By \cref{p:BwP} we may consider $\mathcal B_{w,P}$ as a basis of $H_2(\XX_{w,P},\ZZ)$. Let
\[
\mathcal B^*_{w,P}:=\{\checkmu^*_w(k)\mid k\in I^P_w\} \subset H^2(\XX_{w,P},\ZZ)
\]
denote the dual basis. If $G$ is not simply-laced, but $\XX_{w,P}$ is $\QQ$-factorial, then set $\mathcal B_{w,P}:=\RpwP$. This gives a $\QQ$-basis of $H_2(\XX_{w,P},\QQ)$, see \cref{p:Qcal}, and in the factorial case, $\RpwP$ gives a $\ZZ$-basis of $H_2(\XX_{w,P},\ZZ)$. We again define $\mathcal B^*_{w,P}$ as the dual basis to $\RpwP$. Recall that we identify $H_2(\XX_{w,P},\ZZ)$ with $\mathcal Q_{w,P}$ and $H^2(\XX_{w,P},\ZZ)$ with $\mathcal P_{w,P}$, compare \cref{d:QPw}.
\end{defn}
\begin{defn}\label{d:matrix}
Consider a Schubert variety $\XX_{w,P}$, where $w\in W^P$. If $G$ is simply-laced, we associate to $\XX_{w,P}$ the following square matrix
\begin{eqnarray*}
M_{w,P}&:=&(\langle\omega_i,\checkmu_w(k)\rangle)_{k,i\in I^P_w},
\end{eqnarray*}
where the $\checkmu_w(k)$ are the elements of the set $\mathcal B_{w,P}$ constructed in \cref{p:BwP}. We also define
\begin{eqnarray*}
N_{w,P}&:=&(\langle \checkmu^*_w(k),\checkalpha_i\rangle)_{i,k\in I^P_w}.
\end{eqnarray*}
Equivalently, $N_{w,P}$ is the inverse matrix of $M_{w,P}$, which we note is invertible over $\ZZ$ by \cref{p:BwP}. The columns of $N_{w,P}$ are the coordinates of $\checkmu_w^*(k)$ in terms of the basis $\{\omega_i\mid i\in I^P_w\}$ of $\mathcal P_{w,P}$.
\end{defn}
\begin{remark} As a consequence of \cref{p:BwP}, we have that $M_{w,P}$ is  obtained from $M_{w,B}$ simply by deleting the rows and columns indexed by $j\in I_P$.
\end{remark}
\begin{defn}[Simply-laced case $\hat{\mathbf n}_{w,P}$]\label{d:nwP} We now use the basis $\mathcal B_{w,P}=\{\checkmu_w(k)\mid k\in I^P_w\}$ to define a vector $\hat {\mathbf n}_{w,P}=(\hat n_k)_{k}$ in $\ZZ^{I^P_w}$. To $\checkmu_w(k)$ associate its height $\operatorname{ht}(\checkmu_w(k))=\langle\rho,\checkmu_w(k)\rangle$, and let $\mathbf h=(\operatorname{ht}(\checkmu_w(k)))_{k\in I^P_w}$. Also let $\mathbf 1=(1)_{k\in I^P_w}$. Then
\begin{equation}\label{e:nhatP}
\hat {\mathbf n}_{w,P}:=N_{w,P}(\mathbf h+\mathbf 1).
\end{equation}
If $P=B$ the expression for $\hat{\mathbf n}_{w,P}$ simplifies. Namely, in this case we have $\sum_{i\in I^B_w}\langle\omega_i,\checketa\rangle=\langle\rho,\checketa\rangle$ for any $\checketa\in \RpwB$, using that $\rho=\sum_{i\in I}\omega_i$ and $\langle\omega_i,\checketa\rangle=0$ whenever $i\notin I^B_w$. Therefore $\mathbf h=M_{w,B}(\mathbf 1)$ and
\begin{equation}\label{e:nhatB}
\hat {\mathbf n}_{w,B}=N_{w,B}( M_{w,B}(\mathbf 1)+\mathbf 1)=\mathbf 1+N_{w,B}(\mathbf 1).
\end{equation}
Thus, for a Schubert variety in $G/B$, the entry $\hat n_i$ of $\hat {\mathbf n}_{w,B}$ is simply $1$ plus the sum of the row of $N_{w,B}$ indexed by $i$. We write $\mathbf n_{w,B}=(n_i)_{i\in I^B_w}$ for the vector $N_{w,B}(\mathbf 1)=\hat {\mathbf n}_{w,B}-\mathbf 1$.
\end{defn}

\begin{remark}
In general, for $\checketa\in \RpwP$ with $P\ne B$ there are two natural notions of `height'. One is the standard one used above, $\height(\checketa)=\langle\rho,\checketa\rangle$. But the more natural from the geometric perspective is $\height_P(\checketa):=\sum_{i\in I^P_w}\langle \omega_i,\checketa\rangle$. Namely, this one reflects the degree of the curve class in $H_2(\XX_{w,P})$ associated to $\checketa$ in the minimal projective embedding of $\XX_{w,P}$. We clearly have $\height_P(\checketa)\le\height(\checketa)$ and the vector $\mathbf h_P=(\height_P(\checkmu_w(k)))_{k\in I^P_w}$ is related to $M_{w,P}$ by taking row sums, $\mathbf h_P=M_{w,P}(\mathbf 1)$. If  the difference between these two versions of `height' for $\checkmu_w(k)\in \mathcal B_{w,P}$ is denoted $d_{k}=\langle\sum_{j\in I_P}\omega_j,\checkmu_w(k)\rangle$ and $\mathbf d=(d_k)_{k\in I^P_w}$, then the generalisation of \eqref{e:nhatB} using $\mathbf h=\mathbf h_P+\mathbf d$ is
\[
\hat {\mathbf n}_{w,P}=N_{w,P}( M_{w,P}(\mathbf 1)+\mathbf d+\mathbf 1)=\mathbf 1+N_{w,P}(\mathbf d +\mathbf 1).
\]
Note that $\mathbf d$ captures precisely the part of the height of $\checketa_w(k)$ which is also affected by the $P$-adaptation process, compare \cref{l:M}.
\end{remark}

\begin{prop}\label{p:anticanon}  Suppose $G$ is of simply-laced type.
The following conditions are equivalent
\begin{enumerate}
\item $\XX_{w,P}$ is Gorenstein.
\item $\left(\sum_{k\in I^P_w}\hat n_{k}\langle\omega_k,\checketa\rangle\right)-\operatorname{ht}(\checketa)= 1$ for all $\checketa\in \RpwP\setminus\mathcal B_{w,P}$.
\end{enumerate}
If $\XX_{w,P}$ is Gorenstein, then the anticanonical line bundle on $\XX_{w,P}$ is given by
$\mathcal L_{\lambda}|_{\XX_{w,P}}$,
where $\lambda=\sum_{i\in I^P_w}\hat n_i\omega_i$. In particular, $\lambda$ represents the first Chern class $c_1(\XX_{w,P})$.

Moreover, in the $P=B$ case, the condition (2) can be replaced by
\begin{enumerate}
\item[(2')] $\sum_{i\in I^B_w}  n_{i}\langle\omega_i,\checketa\rangle= 1$ for all $\checketa\in \RpwB\setminus\mathcal B_{w,B}$,
\end{enumerate}
where $n_i=\hat n_i-1$ as in \cref{d:nwP}.
\end{prop}

\begin{remark}\label{r:canonical} For general $w\in W^P$, the numbers $\hat n_i$ are non-canonical, depending on the construction of $\mathcal{B}_{w, P}$. However, whenever $\XX_{w,P}$ is Gorenstein the sum $\sum_{i\in I^P_w}\hat n_k\omega_k$ will not depend on the choices made in the construction of the basis $\mathcal B_{w,P}$ and $\hat{\mathbf n}_{w,P}$. See also \cref{r:PisB}. The condition (2') is the condition \eqref{e:IntroBGorenstein} from the introduction written out in coordinates.
\end{remark}

For $\QQ$-factorial $\XX_{w,P}$ of general Lie type we may describe the Gorenstein condition in the similar terms.

\begin{defn}[$\mathbb Q$-factorial case]\label{d:matrixQfac}
Suppose $\XX_{w,P}$ is $\QQ$-factorial, now of general type. We have a square, integer matrix associated to $\XX_{w,P}$,
\[
M_{w,P}:=\left(\langle \omega_k,\checketa\rangle\right)_{\checketa\in \mathcal B_{w,P},k\in I^P_w},
\]
which is invertible over $\QQ$, with $\mathcal B_{w,P}=\RpwP$ as in \cref{d:dualbasis}. We denote by
\[
N_{w,P}=(n_{k,\checketa})_{k\in I^P_w,\checketa\in\RpwP}
\]
the inverse of $M_{w,P}$ (over $\QQ$ in general). In simply-laced type we may index the elements of $\RpwP$ by $k\in I^P_w$ as in \cref{p:BwP} so that $M_{w,P}$ and $N_{w,P}$ agree with the matrices from \cref{d:matrix} in that case. We let $\hat {\mathbf n}_{w,P}:=N_{w,P}(\mathbf h+\mathbf 1)$, where $\mathbf h$ is the vector of heights $(\operatorname{ht}(\checketa))_\checketa$ and $\mathbf 1=(1,\dotsc, 1)$, as in \cref{d:nwP}. Here this formula defines an element of $\QQ^{\RpwP}$.
\end{defn}
\begin{prop}\label{p:anticanon2}
Suppose $\XX_{w,P}$ is $\QQ$-factorial (of arbitrary Lie type), and let $\hat{\mathbf n}_{w,P}$ be the vector from \cref{d:matrixQfac} with entries $\hat n_{k}\in\QQ$. Then the anticanonical class is given by
\[
[-K_{\XX_{w,P}}]=\sum_{k\in I^P_w} \hat n_k[\XX^{s_k}_{w,P}],
\]
and thus $\XX_{w,P}$ is Gorenstein if and only if $\hat n_k\in\ZZ$ for all $k\in I^P_w$.
\end{prop}

\begin{remark}\label{r:PisB}
In the situation of \cref{p:anticanon2} or \cref{p:anticanon}, but with $P=B$, we let $n_i=\sum_{\checketa\in \RpwB}n_{i,\checketa}$ be the row sum of $N_{w,B}$ for the row labeled $i$, as in \cref{d:nwP}. If $\XX_{w,P}$ is $\QQ$-Gorenstein then the anticanonical class of $\XX_{w,B}$ is given by
\[
[-K_{\XX_{w,B}}]=\sum_{i\in I^B_w} (n_i+1)[\XX^{s_i}_{w}].
\]
Moreover, $\XX_{w,B}$ is Gorenstein if and only if the row sums $n_i$ of $N_{w,B}$ lie in $\ZZ$ for all $i\in I^B_w$. In this case we also have a natural description of the first Chern class of $\XX_{w,B}$ under the identification of $H^2(\XX_{B,w},\ZZ)$ with $\mathcal P_{w,B}$. Namely, $\mathcal P_{w,B}$ has two bases, $\{\omega_i\mid i\in I_w\}$ and the dual basis to $\mathcal B_{w,B}$ that we may denote as $\{\checketa^*_w(i)\mid i\in I_w^B\}$. The first Chern class is equal to the sum of these two bases, $c_1(\XX_{w,B})=\sum_{i\in I_w}\omega_i+\sum_{i\in I_w}\checketa^*_w(i)$. Note that here, despite the $\checketa^*_w(i)$ being non-canonical, their sum is canonical. For example the Schubert variety in $SL_5/B$ associated to the permutation 53142 discussed in \cref{ex:A4nonunique} has two different possible choices for $\mathcal B^*_{w,B}$. But the sum of the dual basis elements is the same in either choice. This Schubert variety is indeed Gorenstein as can be verified using \cref{p:anticanon}, but also independently using \cite{WooYong}.
\color{blue}
\end{remark}

The two propositions, \cref{p:anticanon} and \cref{p:anticanon2}, can now be proved together.

\begin{proof}[{Proof of \cref{p:anticanon} and \cref{p:anticanon2}}] First, let us express the anticanonical divisor class $[-K_{\XX_{w,P}}]$ in terms of the Schubert basis. We have in $G/B$ that for any $j\in I_w^B$,
\[
[\XX^{s_j}_{w,B}]=[\XX^{s_j,B}].[\XX_{w,B}]=
\sum_{\checketa\in \RpwB} \langle \omega_j,\checketa\rangle [\XX_{ws_\checketa ,B}],
\]
using the Chevalley formula as in the proof of  \cref{p:PicardGroup}, see \eqref{e:DiviaSchubCalc}. Thus, for the projected Richardson variety, $\XX^{s_j}_{w,P}=\pi(\XX^{s_j}_{w,B})$ we have in $H_{2\ell(w)-2}(\XX_{w,P},\ZZ)$ that
\[
[\XX^{s_j}_{w,P}]=\pi_*[\XX^{s_j}_{w,B}]=
\sum_{\checketa\in \RpwP} \langle \omega_j,\checketa\rangle [\XX_{ws_\checketa ,P}],
\]
using that $\pi_*[\XX_{ws_\checketa,B}]=0$ if $\checketa\in \RpwB\setminus\RpwP$ for dimension reasons. Note however that unlike in \cref{p:PicardGroup}, the divisor $\XX^{s_j}_{w,P}$ may not be a Richardson variety in $\XX_{w,P}$ (or Cartier), since $j$ is now not assumed to lie in $I^P_w$.

Recall that $[-K_{\XX_{w,P}}]$ involves the sum of all $[\XX^{s_j}_{w,P}]$ for $j\in I^B_w$. We have that
$\sum_{j\in I^B_w}\langle\omega_j,\checketa\rangle=\langle\rho,\checketa\rangle=\operatorname{ht}(\checketa)$, for any $\checketa\in \RpwB$. Therefore
\begin{equation}\label{e:-KXproof}
[-K_{\XX_{w,P}}]=\sum_{j\in I^B_w}[\XX^{s_j}_{w,P}] +\sum_{\checketa\in \RpwP}[\XX_{ws_\checketa,P}]=\sum_{\checketa\in \RpwB} (\operatorname{ht}(\checketa) +1) [\XX_{ws_\checketa ,P}].
\end{equation}
We have that $\XX_{w,P}$ is Gorenstein if and only if $-K_{\XX_{w,P}}$ is Cartier.
Assuming $-K_{\XX_{w,P}}$ is Cartier we make an ansatz for the associated line bundle.
Namely we set $\mathcal O(-K_{\XX_{w,P}})=\mathcal L_{\mu}|_{\XX_{w,P}}$ where $\mu=\sum_{k\in I^P_w} m_k\omega_k$ is to be determined. Then we compute the associated divisor class as in the proof of \cref{p:PicardGroup} and separate out the summands corresponding to $\checkmu_w(k)\in \mathcal B_{w,P}$,
\begin{equation*}
\div(\mathcal L_{\mu}|_{\XX_{w,P}})=
\sum_{\checketa\in \RpwP}\langle \mu, \checketa \rangle[\XX_{w s_\checketa,P}] =
\sum_{k\in I^P_{w}}\langle \mu, \checkmu_w(k) \rangle[\XX_{w s_{\checkmu_w(k)},P}]+\sum_{\checketa\in \RpwP\setminus\mathcal B_{w,P}}\langle \mu, \checketa \rangle[\XX_{w s_\checketa,P}].
\end{equation*}
Comparing the formula for $\div(\mathcal L_{\mu}|_{\XX_{w,P}})$ with the one for $-K_{\XX_{w,P}}$ in \eqref{e:-KXproof} and projecting onto the span of the Schubert basis elements $[\XX_{w s_{\checkmu_w(k)},P}]$ corresponding the $\checkmu_w(k)\in \mathcal B_{w,P}$,
 we have
\[
\sum_{k\in I^P_{w}}\langle \mu, \checkmu_w(k) \rangle[\XX_{w s_{\checkmu_w(k)},P}]=
\sum_{k\in I^P_w} (\operatorname{ht}(\checkmu_w(k)) +1) [\XX_{ws_{\checkmu_w(k)} ,P}].
\]
In terms of the coefficient vector $\mathbf m=(m_k)_{k\in I^P_{w}}$ of $\mu=\sum_{k\in I^P_w} m_k\omega_k$, this equation can be rephrased as $M_{w,P}(\mathbf m)=\mathbf h+\mathbf 1$, compare \cref{d:nwP}. Thus $\mu$ is uniquely determined by just these summands, and has coefficients given by $\mathbf m=N_{w,P}(\mathbf h+\mathbf 1)=\hat{\mathbf n}_{w,P}=(\hat n_k)_k$.

The condition for $\mathcal L_{\mu}|_{\XX_{w,P}}$ to be the anticanonical line bundle is now that for any remaining $\checketa\in \RpwP\setminus\mathcal B_{w,P}$ we have
\begin{equation}\label{e:GorensteinCrit1B}
\sum_{k\in I^P_w}\hat n_{k}\langle\omega_k,\checketa\rangle= \operatorname{ht}(\checketa)+1.
\end{equation}
This translates to the condition (2) in \cref{p:anticanon}. Moreover, condition (2) is equivalent to (2') in the $P=B$-case by \eqref{e:nhatB}. The formula for $[-K_{\XX_{w,P}}]$ and the characterisation of the Gorenstein property in \cref{p:anticanon} and \cref{p:anticanon2} also follow.
 \end{proof}

\begin{cor}\label{c:Fano}
Let $\XX_{w,P}$ be either a factorial Schubert variety of arbitrary Lie type or a simply-laced Gorenstein Schubert variety, and let
$N_{w,P}$ 
and $\hat{\mathbf n}_{w,P}=(\hat n_k)_{k\in I^P_w}$ be as in \cref{d:matrix}, \cref{d:nwP} and \cref{d:matrixQfac}.
The following are equivalent.
\begin{enumerate}
\item $\XX_{w,P}$ is a Gorenstein Fano variety.
\item $\hat{\mathbf n}_{w,P}\in \ZZ_{>0}^{I^P_w}$.
\end{enumerate}
Moreover if $P=B$ then $\XX_{w,B}$ is Fano if and only if the row-sum vector $N_{w,B}(\mathbf 1)\ge 0$.
\end{cor}
\begin{proof}
By \cref{p:anticanon} and \cref{p:anticanon2} the anticanonical line bundle is $\mathcal L_{\lambda}|_{\XX_{w,P}}$ with $\lambda=\sum_{k\in I_w^P} \hat n_k\omega_k$. Let us assume that $w$ has full support, so that $I_w^P=I^P$. Otherwise we may replace $G$ with a Levi subgroup. If all $\hat n_k>0$ then $\mathcal L_{\lambda}$ is ample on $G/P$, and therefore its restriction $\mathcal L_{\lambda}|_{\XX_{w,P}}$ is also ample. Conversely, if  $\mathcal L_{\lambda}|_{\XX_{w,P}}$ is ample, then for any $k\in I^P$ the restriction to $\XX_{s_k,P}\subset \XX_{w,P}$ is ample (and equals $\mathcal O_{\XX_{s_k,P}}(\hat n_k)$). This implies $\hat n_k>0$. The characterisation in the $P=B$ case follows from the second half of \cref{d:nwP}.
\end{proof}

We note that \textbf{\cref{mt:Gorenstein}} and \textbf{\cref{mt:factorialFano}} are the combination of \cref{p:anticanon},  \cref{p:anticanon2} and \cref{c:Fano}.

\section{Examples}\label{s:Examples}
Recall that we simply denote a coroot $\checkalpha_{i_1}+\checkalpha_{i_2}+\cdots+\checkalpha_{i_k}$ as  $\checkalpha_{i_1i_2\ldots i_k}$.
\begin{example}\label{ex:Bw32Fano} Consider $G/B$ of type $A_4$ and let $w=s_3 s_4s_1s_2s_3$. Below we show that the (singular) Schubert variety $\XX_{w, B}$ in $G/B$ is  Gorenstein Fano but is not  factorial.

We have that every element of $\check R^+(w)$ is indecomposable, so that
\[
\RpwB=\check R^+(w)=\{\checkalpha_3,\checkalpha_{23},\checkalpha_{123},\checkalpha_{34},\checkalpha_{234}\}.
\]
Meanwhile $I^B_w=I=[4]$ and $\Pic(\XX_{w, B})\cong \ZZ^4$. So   $\XX_{w, B}$ is not factorial by Theorem \ref{mt:factorialgen}.

We now construct the subset $\mathcal B_{w,P}$ of $\RpwP$, by for each $i\in I^B_w$ selecting the rightmost occurrence of $s_i$ in any reduced expression of $w$ and letting $\checkmu_w(i)$ be the associated positive root. This gives us
\begin{equation}\label{e:Bw32}
\mathcal B_{w,B}=\{\checkmu_w(3)=\checkalpha_3,\checkmu_w(2)=\checkalpha_{23},\checkmu_w(1)=\checkalpha_{123},\checkmu_w(4)=\checkalpha_{34}\}
\end{equation}
We now order the indexing set as follows $I^B_{w}=\{3,2,1,4\}$. This ordering is compatible with the position of the root $\checkmu_w(k)$ in its optimal reflection ordering. The entries of the matrix $M_{w,B}=(\langle \omega_i,\checkmu_w(k)\rangle)_{k,i}$ can be read off row by row, directly from \eqref{e:Bw32}, and we have
\[
M_{w,B}=\begin{pmatrix}
1 & 0 & 0&0\\
1 &1 & 0& 0\\
1 & 1& 1 &0\\
1 & 0 &0 &1
\end{pmatrix},\qquad
N_{w,B}=\begin{pmatrix}
1 & 0 & 0&0\\
-1 &1 & 0& 0\\
0 & -1& 1 &0\\
-1 & 0 &0 &1
\end{pmatrix},
\]
where $N_{w, B}$ is the inverse of $M_{w,B}$. We may read off the dual basis $\mathcal B^*_{w,B}$ from the columns of $N_{w,B}$ (keeping in mind the ordering $3,2,1,4$ of $I^B_w$). Namely,
\[
\mathcal B^*_{w,B}=\{\checkmu^*_w(3)=\omega_3-\omega_2-\omega_4,\, \checkmu^*_w(2)=\omega_2-\omega_3,\, \checkmu^*_w(1)=\omega_{1},\, \checkmu^*_w(4)=\omega_{4}\}.
\]
We now have $\sum_{k}\checkmu_w^*(k)=\omega_3$ and we see that $\XX_{w,B}$ is Gorenstein, because
 $\RpwB\setminus \mathcal B_{w,B}=\{\checkalpha_{234}\}$ and  $\langle\omega_3,\checkalpha_{234}\rangle=1$.
We have $\hat{\mathbf n}_{w,B}=\mathbf 1+N_{w,B}(\mathbf 1)=(2,1,1,1)$ implying Fano by \cref{c:Fano}.
The anticanonical line bundle is given by
\[\mathcal L_{2\omega_3+\omega_2+\omega_1+\omega_4}|_{\XX_{w,B}}.
\]
\end{example}
\begin{example}\label{ex:Pw32Fano}
Consider $G/P$ of type $A_4$ with $I^P=\{1,2,3\}$ and $I_P=\{4\}$. Then we consider the same Weyl group element  $w=s_3s_4s_1s_2s_3$, now as element of $W^P$, and below we show that the Schubert variety $\XX_{w, P}$ is both Gorenstein Fano and factorial. Since $ws_3\notin W^P$ and $ws_3s_2s_3\notin W^P$, we must  remove the   elements $\checkalpha_3,\checkalpha_{23}$ from $\RpwB$ to obtain
\[
\RpwP=\{\checkalpha_{123},\checkalpha_{34},\checkalpha_{234}\}.
\]
The two   elements  $\checkmu^B_w(1)=\checkalpha_{123}$ and $\checkmu^B_w(4)=\checkalpha_{34}$ of $\RpwP$ appeared in  $\mathcal B_{w,B}$ in \cref{ex:Bw32Fano},  but the third didn't. We now demonstrate how the recursive construction of $\mathcal B_{w,P}$ from \cref{d:BwP} gives us $\check R^+_{w,P}$ (as it must).

We begin with $\checkmu=\checkmu^B_w(3)$ which does not have $ws_{\checkmu}(\checkalpha_4)>0$ (and therefore $ws_{\checkmu}\notin W^P$). The recursion tells us to replace $\checkmu^B_w(3)$ by  $s_{\checkmu^B_w(3)}(\checkalpha_4)=\checkmu^B_{w}(3)+\checkalpha_4=\checkalpha_{3}+\checkalpha_{4}$, and this now lies in $\RpwP$. So we set
\[
\checkmu^P_w(3):=\checkalpha_{3}+\checkalpha_{4}.
\]
Next we note that $\checkmu=\checkmu^B_w(2)=\checkalpha_{23}$ does not have $ws_{\checkmu}(\checkalpha_4)>0$, so we analogously replace it by
\[
\checkmu^P_w(2):=\checkalpha_{2}+\checkalpha_{3}+\checkalpha_{4}.
\]
The remaining element (indexed by $I^P_w$) is $\checkmu^B_w(1)$, which does lie in $\RpwP$ already, and so does not need to be replaced. Thus altogether we have
\begin{equation}\label{e:Pw32}
\mathcal B_{w,P}=\{\, \checkmu^P_w(3)=\checkalpha_{34},\ \checkmu^P_w(2)=\checkalpha_{234},\ \checkmu^P_w(1)=\checkalpha_{123}\, \}.
\end{equation}
This illustrates the recursive procedure.
Note that in this example the cardinality of $\RpwP$ was already minimal, equal to the rank $|I^P_w|=3$ of $\Pic(\XX_{w,P})$. Therefore, as a set we were bound to arrive at $\mathcal B_{w,P}=\RpwP$ in this case. Unlike the full flag variety example $\XX_{w,B}$ from above, the Schubert variety $\XX_{w,P}$ in $G/P$ is factorial. In fact, $\XX_{w,P}$ in $G/P$ is a smooth Schubert variety by \cite{GasharovReiner}.

We read off the entries of the matrix $M_{w,P}=(\langle \omega_i, \checkmu_w(k)\rangle)_{k,i}$ now with $i,k\in I^P_w$, row by row directly from \eqref{e:Pw32}. This gives
\[
M_{w,P}=\begin{pmatrix}
1 & 0 & 0\\
1 &1 & 0\\
1 & 1& 1
\end{pmatrix},\qquad
N_{w,P}=\begin{pmatrix}
1 & 0 & 0\\
-1 &1 & 0\\
0 & -1& 1
\end{pmatrix},
\]
where $N_{w, P}$ is the inverse of $M_{w,P}$. We may read off the dual basis $\mathcal B^*_{w,P}$ from the columns of $N_{w,P}$ (keeping in mind the ordering $3,2,1$ of $I^P_w$). Namely,
\[
\mathcal B^*_{w,P}=\{\checkmu^*_w(3)=\omega_3-\omega_2,\, \checkmu^*_w(2)=\omega_2-\omega_1,\, \checkmu^*_w(1)=\omega_{1}\}.
\]
We also have $\hat{\mathbf n}_{w,P}=N_{w,P}(\mathbf h+\mathbf 1)=(3,1,0)$. Using $\hat{\mathbf n}_{w,P}$ and \cref{p:anticanon}, or equivalently using \eqref{e:c1B}, we find that
$$
c_{1}(\XX_{w, P})=\sum_{k=1}^3(\mbox{ht}(\checkmu_w(k))+1)\mu_w^*(k)= 3(\omega_3-\omega_2)+4(\omega_2-\omega_1)+4\omega_1=3\omega_3+\omega_2.$$
Hence $\XX_{w,P}$ is semi-Fano with anticanonical line bundle given by
\[
\mathcal L_{3\omega_3+\omega_2}|_{\XX_{w,P}}.
\]
\end{example}
\begin{example} The following Schubert variety in $G/B$ of type $D_4$ is not Gorenstein.
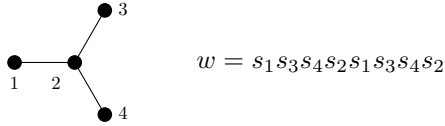
\begin{center}\dynkin[edge length=.8cm,
root radius=.09cm,label]{D}{4}$\qquad {w=s_{1}s_3s_{4}s_2s_{1}s_3s_{4}s_2}$
\end{center}
We compute first the set $\RpwB$ to be
\[
\RpwB=\{\checkalpha_2,\checkalpha_{24},\checkalpha_{23},\checkalpha_{12},\checkalpha_1+2\checkalpha_2+\checkalpha_3+\checkalpha_4,\checkalpha_{123},\checkalpha_{124},\checkalpha_{234}\}.
\]
Within that we have the basis
\[
\mathcal B_{w,B}=\{\checkmu_w(2)=\checkalpha_2,\ \checkmu_w(1)=\checkalpha_{12},\ \checkmu_w(3)=\checkalpha_{23},\ \checkmu_w(4)=\checkalpha_{24}\}.
\]
The dual basis elements are given by
\[
\checkmu_w^*(2)=\omega_2-\omega_1-\omega_3-\omega_4,\ \  \checkmu_w^*(1)=\omega_1,\ \ \checkmu_w^*(3)=\omega_3,\ \  \checkmu_w^*(4)=\omega_4
\]
The sum $\sum_{k}\checkmu_w^*(k)=\omega_2$. However, the highest coweight, $\check\theta=\checkalpha_1+2\checkalpha_2+\checkalpha_3+\checkalpha_4$ is an element of $\RpwB\setminus \mathcal B_{w,B}$, and it has the property $\langle \sum_{k}\checkmu_w^*(k), \check\theta\rangle=\langle \omega_2,\check \theta\rangle=2$. Therefore the Schubert variety $\XX_{w,B}$ is not Gorenstein by the formulation \eqref{e:IntroBGorenstein} of Theorem \ref{mt:Gorenstein}. Note that the condition (2') from \cref{p:anticanon} amounts to the exact same calculation.
\end{example}
\begin{example}[Schubert surfaces and $3$-folds in $G/B$ for type $G_2$]\label{ex:G2surfacesB}
Let $G$ be of type $G_2$ with   $\alpha_1$ being the long simple root  as in \cref{ex:G2}. Recall that $M_{w,B}=\big(\langle\omega_i,\checketa\rangle\big)$ has columns corresponding to $\omega_i$ indexed by $I_w^B$, and rows corresponding to $\checketa$ from $\RpwB$. Both of these sets have cardinality $2$, where we are assuming $w$ is of length $2$ or $3$ in $W$ of type $G_2$. Recall that $N_{w, B}$ denotes the inverse of $M_{w, B}$.
\vskip.2cm
\noindent {1.) $\XX_{s_1s_2,B}$.} For $w={s_1s_2}$, we have $\RpwB=\{\checkalpha_2,\checkalpha_1+\checkalpha_2\}$. We order $I_{w}^{B}=\{2,1\}$.
\[
M_{s_1 s_2,B}=\begin{pmatrix} 1 &0\\
1&1
\end{pmatrix},\quad
N_{s_1 s_2,B}=\begin{pmatrix} 1 &0\\
-1&1
\end{pmatrix}, \quad
N_{s_1 s_2,B}\begin{pmatrix}1\\1\end{pmatrix}=\begin{pmatrix}1\\0\end{pmatrix}.\quad
\]
Thus, this Schubert variety is Fano. Keeping in mind the ordering of $I_w^B$, we have $n_2=1$ and $n_1=0$, and the first Chern class of the anticanonical line bundle is $c_1(\XX_{s_1s_2,B})=\omega_1+2\omega_2$.
\vskip.2cm
\noindent {2.) $\XX_{s_2s_1,B}$.} For $w={s_2s_1}$, we have $\RpwB=\{\checkalpha_1,3\checkalpha_1+\checkalpha_2\}$. We order $I_{w}^{B}=\{1,2\}$.
\[
M_{s_2 s_1,B}=\begin{pmatrix} 1 &0\\
3&1
\end{pmatrix},\quad
N_{s_2 s_1,B}=\begin{pmatrix} 1 &0\\
-3&1
\end{pmatrix},\quad
N_{s_2 s_1,B}\begin{pmatrix}1\\1\end{pmatrix}=\begin{pmatrix}1\\-2\end{pmatrix}.\quad
\]
Thus $\XX_{s_2s_1,B}$ is no longer Fano. Indeed, we have  $c_1(\XX_{s_2s_1,B})=2\omega_1-\omega_2$.

\vskip.2cm
\noindent {3.) $\XX_{s_1s_2s_1,B}$.}  For $w={s_1s_2s_1}$, we order $I_{w,B}=\{1,2\}$. In this case,
 \[
\check R^+_{s_{1}s_2s_{1},B}=\{\checkalpha_{1},\ 2\checkalpha_1+ \checkalpha_2 \}\subset \check R^+(s_{1}s_2s_{1})=\{\checkalpha_{1},\   3\checkalpha_1+\checkalpha_2,\ 2\checkalpha_1+ \checkalpha_2 \},
\]
\[
M_{s_1s_2 s_1,B}=\begin{pmatrix} 1 &0\\
2&1
\end{pmatrix},\quad
N_{s_1s_2 s_1,B}=\begin{pmatrix} 1 &0\\
-2&1
\end{pmatrix}, \quad
N_{s_1s_2 s_1,B}\begin{pmatrix}1\\1\end{pmatrix}=\begin{pmatrix}1\\-1\end{pmatrix}.\quad
\]
 Hence, $\XX_{s_1s_2s_1, B}$ is factorial, for $M_{s_1s_2s_1, B}$ being   invertible over $\ZZ$.  Moreover, $\XX_{s_1s_2s_1,B}$ is not Fano (but only semi-Fano), with   $c_1(\XX_{s_1s_2s_1,B})=2\omega_{1}$.

\vskip.2cm

\noindent {4.) $\XX_{s_2s_1s_2,B}$.} For $w={s_2s_1s_2}$, we order $I_{w,B}=\{2,1\}$. In this case,
 \[
\check R^+_{s_2s_1s_2,B}=\{\checkalpha_{2},\ 3\checkalpha_1+ 2\checkalpha_2\}\subset \check R^+(s_2s_1s_2)=\{\checkalpha_{2},\  \checkalpha_1+\checkalpha_2, \ 3\checkalpha_1+ 2\checkalpha_2\}.
\]
\[
M_{s_2 s_1 s_2,B}=\begin{pmatrix} 1 &0\\
2&3
\end{pmatrix},\quad
N_{s_2 s_1 s_2,B}=\begin{pmatrix} 1 &0\\
-{2\over 3}&{1\over 3}
\end{pmatrix}, \quad
N_{s_2 s_1 s_2, B}\begin{pmatrix}1\\1\end{pmatrix}=\begin{pmatrix}1\\-{1\over 3}\end{pmatrix}.\quad
\]
Hence,  $\XX_{s_2 s_1 s_2, B}$ is $\mathbb{Q}$-factorial but not  factorial, for $M_{s_2 s_1 s_2, B}$ being   invertible over $\QQ$ other than $\ZZ$.
Moreover, it is $\QQ$-Gorenstein Fano, with the first Chern class as an element of $H^2(\XX_{s_2 s_1 s_2,B},\QQ)$ given by $c_1(\XX_{s_2 s_1 s_2,B})=2\omega_2+\frac{2}3\omega_1$.
In particular,   $\XX_{s_2 s_1 s_2,B}$ is singular.

We note that    $\XX_{s_1s_2s_1, B}$  (resp. $\XX_{s_2s_1s_2, B}$)  was known to be    smooth (resp. singular) by    \cite[Theorem 2.4]{BilleyPostnikov}.
\end{example}

\begin{remark}
It was shown in \cite{Fan, Kar} that a general Schubert variety $X_{w, B}$ in $G/B$ is a (smooth) toric variety if and only if $w$ is a product of distinct simple reflections. A characterisation of the (weak) Fano property of toric Schubert varieties was provided in \cite{LMP}. In Remark 1.5 of loc. cit., we also see that the Schubert variety $\XX_{s_2s_1, B}$ of type $G_2$ is isomorphic to the
Hirzebruch surface $\mathbb{P}(\mathcal{O}\oplus \mathcal{O}(3))$.

In addition,  the Schubert variety $\XX_{s_1s_2s_1, B}$ of type $G_2$ is an example of lowest dimension  among those smooth non-Fano, non-toric Schubert varieties in a partial flag variety $G/P$. For Schubert varieties in the full flag variety $G/B$ of type $A$,  $X_{s_4s_2s_3s_2s_1, B}$ is a smooth non-Fano, non-toric Schubert variety $\XX_{w, B}$, by using similar calculations to the above examples together with      the criterion of the smoothness in \cite{LaSa}.

\end{remark}
\begin{example}[Schubert surfaces in $G_2$-Grassmannians]\label{ex:G2surfacesP} Let $G$ be of type $G_2$ and $P_k$ denote the maximal parabolic subgroup with $I^{P_k}=\{k\}$. Notice that the natural projection $G/B\to G/P_k$ is a $\mathbb{P}^1$-bundle, and the  restriction of the base $G/P_k$ to the Schubert surface    gives a $\mathbb{P}^1$-bundle $X_{s_2s_1s_2, B}\to X_{s_2s_1,P_{1}}$ (resp. $X_{s_1s_2 s_1,B}\to X_{s_1s_2,P_{2}}$) when $k=1$ (resp. $2$).

Notice that $G/P_2$ is isomorphic to the smooth quadratic hypersurface in $\mathbb{P}^6$. Hence, the Schubert surface   $X_{s_1s_2,P_{2}}$ in  $G/P_2$ is isomophic to $\mathbb{P}^2$, being smooth Fano with $c_1(\XX_{s_1 s_2,P_2})=3\omega_2$. Nevertheless, its preimage $X_{s_1s_2 s_1,B}$ is smooth but non-Fano by Example \ref{ex:G2surfacesB}.

It follows directly from the properties of $X_{s_2s_1s_2, B}$  in Example \ref{ex:G2surfacesB} that  the Schubert surface   $X_{s_2s_1,P_{1}}$ in  $G/P_1$
is singular, $\QQ$-Gorenstein and $\QQ$-factorial (since these are local properties and  $X_{s_2s_1s_2, B}\to X_{s_2s_1,P_{1}}$ is a $\mathbb{P}^1$-bundle). Indeed, we can also see these properties from   $R^+_{s_2s_1,P_1}=\{3\checkalpha_1+\checkalpha_2\}$, which leads to $M_{s_2s_1,P_1}=(3)$ and $N_{s_2s_1,P_1}=(\frac 1 3)$. As a consequence,  {  $c_1(\XX_{s_1 s_2,P_2})=\frac{5}3\omega_1$,} so that  $X_{s_2s_1,P_{1}}$ is $\QQ$-Gorenstein Fano.
\end{example}

\end{document}